\documentclass[12pt,reqno]{amsart}
\usepackage{amsmath, amssymb, enumitem, float, mathrsfs,graphicx}
\usepackage{ frcursive, mathrsfs}
\usepackage{hyperref}
\usepackage{tikz-cd}
\usepackage{booktabs}
\usepackage[headings]{fullpage}
 
\usepackage{xcolor}
\usepackage{comment}

 \usepackage{pgfplots}
\pgfplotsset{compat=1.18}

\usepackage[OT2,T1]{fontenc}
\DeclareSymbolFont{cyrletters}{OT2}{wncyr}{m}{n}
\DeclareMathSymbol{\Sha}{\mathalpha}{cyrletters}{"58}

\numberwithin{equation}{section}

\theoremstyle{plain}
\newtheorem{theorem}[subsection]{Theorem}
\newtheorem{lemma}[subsection]{Lemma}
\newtheorem{proposition}[subsection]{{Proposition}}
\newtheorem{corollary}[subsection]{{Corollary}}

\theoremstyle{definition}

\theoremstyle{remark}
\newtheorem{remark}[subsection]{{Remark}}

\newenvironment{salign}{
    \begin{equation}
    \begin{aligned}
}{
    \end{aligned}
    \end{equation}
    \ignorespacesafterend
}

\def\C {{\mathbb C}}
\def\F {{\mathbb F}}

\def\N {{\mathbb N}}

\def\Q {{\mathbb Q}}
\def\R {{\mathbb R}}
\def\Z {{\mathbb Z}}

\def\bb1{ {\mathbb 1}}

\def\cE {{\mathcal E}}

\def\cO {{\mathcal O}}

\def \sL {{\mathscr L}}

\def \a {{\mathfrak a}}

\def \f {{\mathfrak f}}

\def \p {{\mathfrak p}}
\def \q {{\mathfrak q}}

\def \fD {{\mathfrak D}}

\def \bE {{\textnormal{\textbf{E}}}}
\def \bP {{\textnormal{\textbf{P}}}}

\newcommand{\GL}{\textnormal{GL}}

\newcommand{\Frob}{\textnormal{Frob}}
\newcommand{\Gal}{\textnormal{Gal}}

\newcommand{\Ht}{\textnormal{Ht}}
\newcommand{\Ind}{\textnormal{Ind}}

\newcommand{\ord}{\textnormal{ord}}

\newcommand{\RE}{\textnormal{Re}}

\newcommand{\Sym}{\textnormal{Sym}}

\newcommand{\Val}{\textnormal{Val}}

\renewcommand\Re{\operatorname{Re}}
\renewcommand\Im{\operatorname{Im}}

\newcommand{\defeq}{\stackrel{\textnormal{def}}{=}}

\title{Upper bounds for moments of analytic ranks of elliptic curves over number fields}

\author{Tristan Phillips}
\author{Saahil Sharma}
\address{}
\email{}

\begin{document}

\begin{abstract}
    We prove a conditional upper bound for moments of the analytic rank of elliptic curves over number fields. We also prove conditional bounds for the density of elliptic curves with analytic rank larger than a given bound. 
\end{abstract}

\maketitle

\section{Introduction}

Let $K$ be a number field of degree $d$. For $B\in \R_{>0}$, let
$\cE_K(B)$ denote the finite set of isomorphism classes of elliptic curves
over $K$ of naive height\footnote{See Section
\ref{sec:height} for the precise definition of the naive height used in
this paper.} at most $B$. For $m\in \Z_{\geq 1}$, define the $m$-th limsup moment of the
analytic rank in this family by
\begin{equation}\label{eq:intro_moment_def}
    \overline{\bE}_K[r_{\rm an}^m]
    \defeq
    \limsup_{B\to\infty}
    \frac{1}{\#\cE_K(B)}
    \sum_{E\in \cE_K(B)} r_{\rm an}(E)^m .
\end{equation}

The distribution of ranks of elliptic curves is a central problem
in arithmetic statistics. The Birch and Swinnerton-Dyer conjecture predicts
that the analytic rank equals the Mordell--Weil rank, while the minimalist philosophy predicts that, ordered by height, elliptic curves should have rank $0$ with density $1/2$, rank $1$ with density $1/2$, and rank $2$ or greater with density $0$.
Over $\Q$, substantial progress has been made toward average-rank bounds, most notably through the work of Bhargava and Shankar on Selmer groups \cite{BS13b}. Analytic approaches, beginning with Brumer
\cite{Bru92} and developed further by Heath-Brown \cite{HB04}, Young
\cite{You06}, and Baier--Zhao \cite{BZ08}, give conditional bounds for analytic ranks via explicit formula methods. 
Much less is known over general number
fields, where additional difficulties arise (although, see \cite{Sha13} and \cite{Phi25}).

The purpose of this paper is to prove conditional upper bounds for all moments of analytic ranks of elliptic curves over an arbitrary number field. 
Our main theorem is the following.

\begin{theorem}\label{thm:moment_bound}
    Assume the Generalized Riemann Hypothesis for the $L$-functions of
    elliptic curves over $K$ and for their symmetric square $L$-functions, and assume that all elliptic curves over $K$ are modular.
    Then, for every
    $m\in \Z_{\geq 1}$,
    \begin{equation}\label{eq:moment_bound}
       \overline{\bE}_K[r_{\rm an}^m]
        \leq
        \sum_{k=0}^{\lfloor m/2\rfloor}
        \frac{m!}{(m-2k)!\,k!}
        \left(\frac{9dm+1}{2}\right)^{m-2k}
        \left(\frac{1}{6}\right)^k
        \ll
        \left(\frac{9dm}{2}\right)^m .
    \end{equation}
\end{theorem}

Theorem \ref{thm:moment_bound} simultaneously generalizes the theorem of
Cho and Jeong \cite[Theorem 1.3]{CJ23}, which treats the case $K=\Q$, and a
result of the first author \cite[Theorem 1.1.1]{Phi25}, which treats the
first moment over number fields. In the special case $K=\Q$ and $m=1$,
stronger analytic estimates are known: the current best upper bound for the average analytic rank, conditional only on GRH for L-functions, is
$\overline{\bE}_{\Q}[r_{\rm an}]\leq 25/14$, due to Young \cite{You06},
building on work of Brumer \cite{Bru92} and Heath-Brown \cite{HB04}; see
also Baier and Zhao \cite{BZ08} for related bounds under only GRH for elliptic curve $L$-functions. 
For algebraic ranks over $\Q$, Bhargava and Shankar's
work on the average size of $5$-Selmer groups gives the unconditional bound that the average rank is at most $0.885$ \cite{BS13b}.
Moments of ranks of elliptic curves in twists families have been studied in work of Miller and Wong \cite{MW12}. 
%They conditionally prove that the $m$-th moments of analytic ranks of elliptic curve over $\Q$ in any quadratic twist family is $O(m^m)$. 

\begin{remark}
    It is interesting to note that the expression on the right-hand side of \eqref{eq:moment_bound} coincides with the $m$-th moment of the normal distribution with mean $(9dm+1)/2$ and variance $1/3$. It would be interesting to give a proof of Theorem \ref{thm:moment_bound} making use of more probability theory, perhaps using Isserlis's Theorem.
\end{remark}

The proof of Theorem \ref{thm:moment_bound} proceeds through moment bounds for the $1$-level density. We use the explicit formula to majorize analytic rank by a smoothed sum over prime ideals, and then estimate the resulting moments after averaging over elliptic curves ordered by height. 
We are then able to reduce the moment calculation to estimates
for sums of products of Frobenius traces. 
The main arithmetic input is asymptotics for the number of elliptic curves over $K$ satisfying prescribed local conditions. 

As an application of the moment bounds, we obtain quantitative upper bounds
for the density of elliptic curves of large analytic rank. Define the probability
\begin{equation}\label{eq:intro_upper_probability_def}
    \overline{\bP}_K(r_{\rm an}(E)\geq \beta)
    \defeq
    \limsup_{B\to\infty}
    \frac{\#\{E\in \cE_K(B): r_{\rm an}(E)\ge \beta\}}
         {\#\cE_K(B)}.
\end{equation}

\begin{theorem}\label{thm:probability_large_rank}
    Assume the Generalized Riemann Hypothesis for the $L$-functions of all
    elliptic curves over $K$ and for their symmetric square $L$-functions.
    Assume also that all elliptic curves over $K$ are modular. Then, for every
    $m\in \Z_{\geq 1}$ and every $\beta>(1+9dm)/2$,
    \begin{equation}\label{eq:intro_large_rank_bound}
        \overline{\bP}_K(r_{\rm an}(E)\geq \beta)
        \leq
        \frac{(2m)!}{m!}
        \left(
            \frac{2}{3(2\beta-9dm-1)^2}
        \right)^m .
    \end{equation}
    In particular, optimizing over $m$ gives
    \begin{equation}\label{eq:intro_large_rank_asymptotic}
        \overline{\bP}_K(r_{\rm an}(E)\geq \beta)
        \ll_d
        \beta^{-\frac{2\beta}{9d}+o(\beta)},
    \end{equation}
    as $\beta\to\infty$. 
\end{theorem}

Theorem \ref{thm:probability_large_rank} generalizes a theorem of Cho and Jeong \cite[Theorem 1.1]{CJ23}, which proved the case $K=\Q$. For elliptic curves over $\Q$, this improves upon a result of Heath-Brown \cite[Theorem 2]{HB04}, which proves
\begin{equation}
    \overline{\bP}_\Q(r_{an}(E)\geq \beta)\ll \left(\frac{3\beta}{2}\right)^{-\beta/12}.
\end{equation}

\begin{remark}
    The asymptotic analysis in \cite[Remark (2) to Theorem 1.1]{CJ23} leads Cho and Jeong to conclude that their bound is asymptotically weaker than that of Heath-Brown. However, in this article we give a more careful analysis, showing the asymptotic bound \eqref{eq:intro_large_rank_asymptotic}, which is asymptotically stronger than \cite[Theorem 2]{HB04}. 
\end{remark}

Table \ref{tab:OptimalMBoundsD1}, Table \ref{tab:OptimalMBoundsD2}, and Table \ref{tab:OptimalMBoundsD3} give numerical examples of Theorem \ref{thm:probability_large_rank} for several values of $\beta$ and $d$, together with the optimal choice of $m$ in each case.

\begin{table}[ht]
\centering
\begin{tabular}{cccccc}
\hline\vspace{-11pt}\\
$\beta$ & $15$ & $20$ & $25$ & $30$ & $35$ \\
\hline
\vspace{-7pt}\\
optimal $m$ & $2$ & $3$ & $3$ & $4$ & $5$ \\
$\overline{\bP}_{\Q}(r_{\rm an}(E)\geq \beta)$ bound
& $3.64\cdot 10^{-4}$
& $1.19\cdot 10^{-5}$
& $3.14\cdot 10^{-7}$
& $4.24\cdot 10^{-9}$
& $6.28\cdot 10^{-11}$ \\
\hline
\end{tabular}
\caption{Optimal choices of $m$ and approximate bounds from Theorem
\ref{thm:probability_large_rank} when $d=1$.}
\label{tab:OptimalMBoundsD1}
\end{table}

\begin{table}[ht]
\centering
\begin{tabular}{cccccc}
\hline\vspace{-11pt}\\
$\beta$ & $15$ & $20$ & $25$ & $30$ & $35$ \\
\hline
\vspace{-7pt}\\
optimal $m$ & $1$ & $1$ & $2$ & $2$ & $3$ \\
$\overline{\bP}_{K}(r_{\rm an}(E)\geq \beta)$ bound
& $1.10\cdot 10^{-2}$
& $3.02\cdot 10^{-3}$
& $1.87\cdot 10^{-4}$
& $1.91\cdot 10^{-5}$
& $3.12\cdot 10^{-6}$ \\
\hline
\end{tabular}
\caption{Optimal choices of $m$ and approximate bounds from Theorem
\ref{thm:probability_large_rank} when $d=2$.}
\label{tab:OptimalMBoundsD2}
\end{table}

\begin{table}[ht]
\centering
\begin{tabular}{cccccc}
\hline\vspace{-11pt}\\
$\beta$ & $15$ & $20$ & $25$ & $30$ & $35$ \\
\hline
\vspace{-7pt}\\
optimal $m$ & $1$ & $1$ & $1$ & $1$ & $2$ \\
$\overline{\bP}_{K}(r_{\rm an}(E)\geq \beta)$ bound
& $3.33\cdot 10^{-1}$
& $9.26\cdot 10^{-3}$
& $2.75\cdot 10^{-3}$
& $1.30\cdot 10^{-3}$
& $1.05\cdot 10^{-4}$ \\
\hline
\end{tabular}
\caption{Optimal choices of $m$ and approximate bounds from Theorem
\ref{thm:probability_large_rank} when $d=3$.}
\label{tab:OptimalMBoundsD3}
\end{table}

Let us also mention the corresponding unconditional algebraic-rank tail
bound over $\Q$. By Bhargava and Shankar's work determining the average size of the $5$-Selmer group of elliptic curves \cite{BS13b}, together with
Markov's inequality, one obtains
\begin{equation}\label{eq:intro_BS_tail_bound}
    \overline{\bP}_{\Q}(r(E)\geq \beta)
    \leq \frac{6}{5^\beta}.
\end{equation}

\section*{Outline}

In Section \ref{sec:height}, we define the naive height on isomorphism classes of elliptic curves over $K$.
In Section \ref{sec:L-functions}, we review the
analytic properties of the relevant $L$-functions and apply the explicit formula to express moments of the $1$-level density in terms of sums over prime ideals.
In Section \ref{sec:Counting_ECs}, we collect asymptotic counting results for elliptic curves over number fields satisfying prescribed local conditions.
In Section \ref{sec:Frobenius_trace}, we use these counting
results to estimate the Frobenius-trace sums that arise in the moment calculation.
In Section \ref{sec:moment_bound}, we combine these ingredients
to prove Theorems \ref{thm:moment_bound} and
\ref{thm:probability_large_rank}.

\section*{Acknowledgments}

We acknowledge the use of AI tools in the preparation of this manuscript. TP was supported by the National Science Foundation, via grant DMS-2303011.

\section{Naive height}\label{sec:height}

In this short section we introduce some notation used throughout the paper and give a definition for the naive height of an elliptic curve over a number field.

Let $K$ be a number field with ring of integers $\cO_K$, absolute Galois group $G_K\defeq \Gal(\overline{K}/K)$, and discriminant $\Delta_K$. 
Let $\Val(K)$ denote the set of places of $K$, let $\Val_\infty(K)$ denote the set of infinite places of $K$, and let $\Val_0(K)$ denote the set of finite places of $K$. For each $v\in \Val(K)$ let $|\cdot|_v$ denote the corresponding absolute value. For each finite place $v\in \Val_0(K)$,  let $\p_v\subseteq \cO_K$ denote the prime ideal corresponding to $v$, and let $q_v$ denote the size of the residue field of $K$ at $v$.

Let $E$ be an elliptic curve over $K$ with model $E:y^2=x^3+Ax+B$, where $A,B\in \cO_K$. Define
\begin{equation}
    \Ht_v(E)\defeq 
    \begin{cases}
        \max\left\{q_v^{-12\lfloor v(A)/4\rfloor}, q_v^{-12\lfloor v(B)/6\rfloor}\right\} & \text{ if } v\in \Val_0(K),\\
        \max\left\{|A|_v^3, |B|_v^2\right\} & \text{ if } v\in \Val_\infty(K).
    \end{cases}
\end{equation}
Then the \textbf{naive height} of $E$ over $K$ is defined as
\begin{equation}
    \Ht(E)\defeq \prod_{v\in \Val(K)} \Ht_v(E).
\end{equation}
When $K=\Q$ and $E:y^2=x^3+Ax+B$ has the property that $\gcd(A^3,B^2)$ is $12$-th power free, the naive height simplifies to
\begin{equation}
    \Ht(E)=\max\left\{|A|^3, |B|^2\right\}.
\end{equation}

\section{\texorpdfstring{$L$}{}-functions of elliptic curves}\label{sec:L-functions}

In this section we collect the definitions and analytic properties of the $L$-functions attached to elliptic curves that will be used in the paper. We then use the
explicit formula to express the $1$-level density of the low-lying zeros of $L(E/K,s)$ as a smoothed sum over prime ideals.

%---------------------------------------------------------------------
\subsection{The Hasse--Weil \texorpdfstring{$L$}{L}-function}
%---------------------------------------------------------------------

Let $E$ be an elliptic curve over the number field $K$. For each finite place
$v\in \Val_0(K)$, let $E_v$ denote the reduction of $E$ at $v$. If $E$ has
good reduction at $v$, define
\begin{equation}
N_v \defeq \#E_v(\mathbb F_{q_v})
\qquad\text{and}\qquad
a_v(E)\defeq q_v+1-N_v.
\end{equation}
If $E$ has bad reduction at $v$, define
\begin{equation}
b_v\defeq \begin{cases}
1 & \text{ if $E$ has split multiplicative reduction at $v$},\\
-1 & \text{ if $E$ has nonsplit multiplicative reduction at $v$},\\
0 & \text{ if $E$ has additive reduction at $v$}.
\end{cases}
\end{equation}
For each finite place $v\in \Val_0(K)$, define the local polynomial
$L_v(E/K,T)$ by
\begin{equation}
L_v(E/K,T)\defeq\begin{cases}
1-a_v T+ q_v T^2 & \text{ if $E$ has good reduction at $v$},\\
1-b_vT & \text{ if $E$ has bad reduction at $v$}.
\end{cases}
\end{equation}
The \textbf{Hasse--Weil $L$-function} $L(E/K,s)$ of $E$ over $K$ is defined as the Euler product
\begin{equation}
L(E/K,s)\defeq \prod_{v\in \Val_0(K)} L_v(E/K, q_v^{-s})^{-1},
\end{equation}
which converges absolutely for $\RE(s)>\tfrac{3}{2}$.
At a place $v$ of good reduction, let $\alpha_v$ and $\beta_v$ be such that 
\begin{equation}
1-a_v(E)T+q_vT^2=(1-\alpha_vT)(1-\beta_vT),
\end{equation}
so that $\alpha_v+\beta_v=a_v(E)$ and $\alpha_v\beta_v=q_v$.
By the Hasse bound (see, e.g., \cite[\S V Theorem 2.3.1(a)]{Sil09}), we have
\begin{equation}\label{eq:Hasse-bound}
|\alpha_v|=|\beta_v|=q_v^{1/2}.
\end{equation}
Define the \textbf{von Mangoldt function} on non-zero integral ideals
$\a\subseteq \cO_K$ by
\begin{equation}
\Lambda_K(\a)\defeq \begin{cases}
\log(q_v) & \text{ if } \a=\p_v^k \text{ for some } k\geq 1,\\
0 & \text{ otherwise.}
\end{cases}
\end{equation}
For prime powers, define
\begin{equation}\label{eq:aE-hat-def}
\widehat{a_E}(\p_v^k)\defeq
\begin{cases}
    \alpha_v^k+\beta_v^k & \text{ if $E$ has good reduction at $\p_v$,}\\
    b_v^k &  \text{ if $E$ has bad reduction at $\p_v$,}
\end{cases}
\end{equation}
and define $\widehat{a_E}(\a)=0$ whenever $\a$ is not a prime power. Then the logarithmic derivative of $L(E/K,s)$ admits the Dirichlet
series expansion
\begin{align}\label{eq:L-logarithmic_derivative}
\frac{L'}{L}(E/K,s)=-\sum_{\a\subseteq \cO_K} \frac{\widehat{a_E}(\a)\Lambda_K(\a)}{N_{K/\Q}(\a)^s},
\end{align}
valid for $\RE(s)>\tfrac{3}{2}$, where the sum runs over the non-zero integral
ideals of $\cO_K$.

Let $\mathfrak{f}_{E/K}$ denote the conductor of $E$ over $K$, and define
\begin{equation}\label{eq:AEK-def}
A_{E/K}\defeq N_{K/\Q}(\mathfrak{f}_{E/K})\, \Delta_K^2.
\end{equation}
%the (absolute) conductor of $L(E/K,s)$, which governs the analytic conductor.

%---------------------------------------------------------------------
\subsection{The explicit formula and the \texorpdfstring{$1$}{1}-level density}
%---------------------------------------------------------------------

We recall the explicit formula for Hasse--Weil $L$-functions, as stated in
\cite[Corollary 6.1.2]{Phi25} (see also \cite[Lemma 1]{GG07}, \cite[Theorem 5.4]{IK04}, and \cite[Appendix A]{Mil02}).

Let
$\phi:\C\to \C$ be an even function such that there exists an $\epsilon>0$ for which $\phi$ is analytic in the strip
\begin{equation}
    |\Im(s)|\leq \frac12+\epsilon,
\end{equation}
and assume that in this strip
\begin{equation}
    |\phi(s)|\ll (1+|s|)^{-1-\delta}
\end{equation}
for some $\delta>0$ as $|\RE(s)|\to \infty$. We also assume that
$\phi(x)\in \R$ for $x\in \R$. 

Let
\begin{equation}
\widehat{\phi}(t)\defeq
\int_{\R}\phi(x)e^{-2\pi ixt}\,dx
\end{equation}
denote the Fourier transform of $\phi$.

\begin{proposition}[Explicit Formula]\label{prop:modified-expicit-formula}
Let $E$ be a modular elliptic curve over the number field $K$, and suppose
that $L(E/K,s)$ satisfies the Generalized Riemann Hypothesis, so that every
non-trivial zero $\rho_j$ of $L(E/K,s)$ has the form
\begin{equation}
\rho_j=1+i\gamma_j \quad \text{ with }\quad \gamma_j\in \R.
\end{equation}
Let 
\begin{equation}
    r_{an}(E)\defeq\ord_{s=1}L(E/K,s)
\end{equation}
be the analytic rank of $E$.
Then, for a parameter $X>1$,
\begin{align}
\begin{split}\label{eq:modified-explicit-formula}
&r_{an}(E)\phi(0)+\sum_{\gamma_j\neq 0} \phi\left(\gamma_j\frac{\log(X)}{2\pi}\right) \\
&\hspace{1mm}= \frac{\widehat{\phi}(0) \log(A_{E/K})}{\log(X)} -\frac{2}{\log(X)} \sum_{\a\subseteq \cO_K} \frac{ \widehat{a_E}(\a) \Lambda_K(\a)}{N_{K/\Q}(\a)}\cdot\widehat{\phi}\left(\frac{\log(N_{K/\Q}(\a))}{\log(X)}\right)+O_{\phi,K}\left(\frac{1}{\log(X)}\right).
\end{split}
\end{align}
\end{proposition}

The quantity on the left-hand side of \eqref{eq:modified-explicit-formula} is
the \textbf{$1$-level density} of zeros of $L(E/K,s)$ with respect to the test function $\phi$,
\begin{equation}\label{eq:1-level_density}
    D_1(E/K, \phi) \defeq \sum_{\gamma_j} \phi\left(\gamma_j\frac{\log(X)} {2\pi}\right) = r_{an}(E)\phi(0)+\sum_{\gamma_j\neq 0} \phi\left(\gamma_j\frac{\log(X)}{2\pi}\right).
\end{equation}
If $\phi$ is nonnegative on $\R$, then
\begin{equation}
\label{eq:rank_from_density}
r_{\rm an}(E)
\leq
\frac{D_1(E/K,\phi)}{\phi(0)}.
\end{equation}

Writing the ideal sum in \eqref{eq:modified-explicit-formula} as a sum over
prime powers via \eqref{eq:aE-hat-def}, we obtain
\begin{equation}\label{eq:sum_ideals_to_sum_prime_powers}
\sum_{\a\subseteq \cO_K} \frac{ \widehat{a_E}(\a) \Lambda_K(\a)}{N_{K/\Q}(\a)}\cdot\widehat{\phi}\left(\frac{\log(N_{K/\Q}(\a))}{\log(X)}\right)
=\sum_{v\in \Val_0(K)} \sum_{k\geq 1} \frac{ \widehat{a_E}(\p_v^k) \log(q_v)}{q_v^k}\cdot\widehat{\phi}\left(\frac{k \log(q_v)}{\log(X)}\right).
\end{equation}

%---------------------------------------------------------------------
\subsection{The contribution of prime-power terms}
%---------------------------------------------------------------------

We first show that the terms with $k\geq 3$ contribute only $O(1)$ to the sum \eqref{eq:sum_ideals_to_sum_prime_powers}.

\begin{lemma}\label{lem:k>=3}
    %Maintaining the notation of Proposition~\ref{prop:modified-expicit-formula},
    We have
    \begin{equation}\label{eq:k>=3}
        \sum_{v\in \Val_0(K)} \sum_{k\geq 3} \frac{ \widehat{a_E}(\p_v^k) \log(q_v)}{q_v^k}\cdot\widehat{\phi}\left(\frac{k \log(q_v)}{\log(X)}\right) = O_{\phi,K}(1).
    \end{equation}
\end{lemma}

Before giving the proof, we define some notation. Let $m_L$ denote Lebesgue measure on $\R$, and let $||f||_\infty$ denote the usual $L^\infty$-norm 
\begin{equation}
    ||f||_\infty\defeq\inf\left\{a\geq 0 : m_L(\{x\in \R : |f(x)|>a\})=0\right\}.
\end{equation}

\begin{proof}[Proof of Lemma \ref{lem:k>=3}.]
    If $E$ has good reduction at $v$, then by the Hasse bound  \eqref{eq:Hasse-bound}
\begin{equation}
|\widehat{a_E}(\p_v^k)|
=
|\alpha_v^k+\beta_v^k|
\leq 2q_v^{k/2}.
\end{equation}
If $E$ has bad reduction at $v$, then 
\begin{equation}
|\widehat{a_E}(\p_v^k)|=|b_v|^k\leq 1\leq 2q_v^{k/2}.
\end{equation}
Therefore
\begin{equation}
\sum_{v\in \Val_0(K)}
\sum_{k\geq 3}
\left|
\frac{\widehat{a_E}(\p_v^k)\log q_v}{q_v^k}
\widehat{\phi}
\left(
\frac{k\log q_v}{\log(X)}
\right)
\right|
\leq
2\|\widehat{\phi}\|_\infty
\sum_{v\in \Val_0(K)}
\sum_{k\geq 3}
\frac{\log q_v}{q_v^{k/2}}.
\end{equation}
The inner geometric series is bounded by a constant times
$q_v^{-3/2}\log q_v$, and
\begin{equation}
\sum_{v\in \Val_0(K)}\frac{\log q_v}{q_v^{3/2}}=O_K(1).
\end{equation}
\end{proof}

%---------------------------------------------------------------------
\subsection{The \texorpdfstring{$k=2$}{k=2} term and the symmetric square}
%---------------------------------------------------------------------

We now consider the $k=2$ term of the sum \eqref{eq:sum_ideals_to_sum_prime_powers}.

We first observe that finitely many places may always be discarded at the cost
of an $O(1)$ error: for any $k\in \Z_{>0}$ and any finite set of places
$S\subset \Val_0(K)$,
\begin{equation}\label{eq:finite_places_ignorable}
\sum_{v\in S} \frac{ \widehat{a_E}(\p_v^k) \log(q_v)}{q_v^k}\cdot\widehat{\phi}\left(\frac{ k\log(q_v)}{\log(X)}\right)=O(1),
\end{equation}
since the sum is finite and each summand is bounded uniformly in $X$.

At a place $v$ of good reduction, \eqref{eq:aE-hat-def} gives
\begin{equation}\label{eq:a_E(p^2)_identity}
    \widehat{a_E}(\p_v^2)=\alpha_v^2+\beta_v^2=(\alpha_v+\beta_v)^2-2\alpha_v\beta_v=a_v(E)^2-2q_v.
\end{equation}
Applying \eqref{eq:finite_places_ignorable} with $S$ the (finite) set of places
at which $E$ has bad reduction, and using \eqref{eq:a_E(p^2)_identity}, the $k=2$ term in the sum \eqref{eq:sum_ideals_to_sum_prime_powers} equals
\begin{salign}\label{eq:k=2_sum}
    & \sum_{v\in \Val_0(K)} \frac{ \widehat{a_E}(\p_v^2) \log(q_v)}{q_v^2}\cdot\widehat{\phi}\left(\frac{ 2\log(q_v)}{\log(X)}\right)\\
    &\qquad =\sum_{\substack{v\in \Val_0(K)\\ E \text{ good reduction at } v}} \frac{ (a_v(E)^2-2q_v) \log(q_v)}{q_v^2}\cdot\widehat{\phi}\left(\frac{ 2\log(q_v)}{\log(X)}\right) + O(1).
\end{salign}
We write $a_v(E)^2-2q_v = (a_v(E)^2-q_v)-q_v$. The $-q_v$ part is an
elementary main term, treated in Lemma~\ref{lem:k=2_1}; the
$a_v(E)^2-q_v$ part is governed by the symmetric square $L$-function, treated in Proposition~\ref{prop:S2_estimate}.

\begin{lemma}\label{lem:k=2_1}
    Maintaining the notation of Proposition~\ref{prop:modified-expicit-formula},
    we have
    \begin{equation}
        \sum_{v\in \Val_0(K)} \frac{ -q_v \log(q_v)}{q_v^2}\cdot\widehat{\phi}\left(\frac{ 2\log(q_v)}{\log(X)}\right)
        =-\frac{\phi(0)}{4}\log(X) + O_{\phi,K}\left(1\right).
    \end{equation}
\end{lemma}

\begin{proof}
    It suffices to evaluate
    \begin{equation}\label{eq:simlification}
        \sum_{v\in \Val_0(K)} \frac{\log(q_v)}{q_v}\cdot\widehat{\phi}\left(\frac{2\log(q_v)}{\log(X)}\right).
    \end{equation}
    We apply Abel summation with the weights
    $\left(\tfrac{\log(q_v)}{q_v}\right)_{v\in \Val_0(K)}$ and the function 
    \begin{equation}
        f(t)\defeq\widehat{\phi}\!\left(\tfrac{2\log(t)}{\log(X)}\right).
    \end{equation}
    By the prime ideal theorem, the partial sums satisfy
    \begin{equation}
        \Psi(x) \defeq \sum_{q_v\leq x} \frac{\log(q_v)}{q_v} = \log(x) + O_K(1).
    \end{equation}
    Hence
    \begin{salign}
        \sum_{v\in \Val_0(K)} \frac{ \log(q_v)}{q_v}\, f(q_v)
        &= \int_2^\infty f(t)\, d\Psi(t)
        =\int_2^\infty \frac{f(t)}{t}\, dt
        + O_K\left(\int_2^\infty |f'(t)|\, dt\right).
    \end{salign}
    Making the substitution $u=\tfrac{2\log(t)}{\log(X)}$, we have 
    \begin{equation}
        \int_2^\infty \frac{f(t)}{t}\, dt
        = \frac{\log(X)}{2}\int_{u_0}^\infty \widehat{\phi}(u)\, du,
        \quad \text{ where }\ u_0=\tfrac{2\log 2}{\log(X)}=o(1).
    \end{equation}
    Since $\widehat{\phi}$ is even (because $\phi$ is even), integrable, and
    $\int_\R \widehat{\phi}(u)\, du=\phi(0)$ by Fourier inversion, we have
    \begin{equation}
        \int_{u_0}^\infty \widehat\phi(u)\,du = \tfrac12\int_\R\widehat\phi(u)\,du + O_K(u_0)
    = \tfrac{\phi(0)}{2}+O_K\!\left(\tfrac{1}{\log(X)}\right),
    \end{equation}
    and therefore
    \begin{equation}\label{eq:log(X)phi(0)/4}
        \frac{\log(X)}{2}\int_{u_0}^\infty \widehat{\phi}(u)\, du
        = \frac{\phi(0)}{4}\log(X) + O_K(1).
    \end{equation}
    For the error term, the hypotheses on $\phi$ imply the decay
    \begin{equation}\label{eq:phihat_decay}
        |\widehat{\phi}(t)| \ll \exp(-\pi(1+2\epsilon)|t|) \quad \text{ and } \quad 
        |\widehat{\phi}'(t)|\ll \exp(-\pi(1+2\epsilon)|t|). 
    \end{equation}
   Therefore, again substituting $u=\tfrac{2\log t}{\log(X)}$,
    \begin{equation}\label{eq:f'_integral}
        \int_2^\infty |f'(t)|\, dt
        =  \frac{1}{2}\int_{u_0}^\infty |\widehat\phi{}'(u)|\, du
        = O_\phi(1).
    \end{equation}
    Combining equations \eqref{eq:simlification},
    \eqref{eq:log(X)phi(0)/4}, and \eqref{eq:f'_integral}, and multiplying by
    $-1$, gives the desired expression.
\end{proof}

It remains to analyze the sum
\begin{equation}\label{eq:S2-def}
    S_{2}(X)\defeq \sum_{v\in \Val_0(K)} \frac{ (a_v(E)^2-q_v) \log(q_v)}{q_v^2}\cdot\widehat{\phi}\left(\frac{ 2\log(q_v)}{\log(X)}\right),
\end{equation}
which we relate to the symmetric square $L$-function
$L(\Sym^2(E/K),s)$.

For each $v\in\Val_0(K)$ define the local factor
\begin{equation}
    L_v(\Sym^2(E/K), T) \defeq 
    \begin{cases}
        (1-\alpha_v^2 T) (1- \alpha_v\beta_v T) (1-\beta_v^2 T) & \text{ if $E$ has good reduction at $v$,}\\
        (1-T)(1-b_v T) & \text{ if $E$ has mult. reduction at $v$,}\\
        1-\epsilon_v T & \text{ if $E$ has add. reduction at $v$,}
    \end{cases}
\end{equation}
where $\epsilon_v\in \{\pm 1\}$.\footnote{The sign $\epsilon_v$ will not be consequential for our analysis, and thus we refrain from giving its precise definition.}
The symmetric square $L$-function is the Euler product
\begin{equation}
L(\Sym^2(E/K),s)
\defeq
\prod_{v\in \Val_0(K)}
L_v(\Sym^2(E/K),q_v^{-s})^{-1},
\end{equation}
which converges absolutely for $\RE(s)>2$.

Define the coefficients
$\lambda_{\Sym^2}(\a)$ by the logarithmic derivative expansion
\begin{equation}\label{eq:L'/L_Sym^2}
    -\frac{L'}{L}(\Sym^2(E/K),s)
    =
    \sum_{\a\subseteq\cO_K}
    \frac{\lambda_{\Sym^2}(\a)\,\Lambda_K(\a)}
    {N_{K/\Q}(\a)^s}.
\end{equation}
At a non-zero prime ideal $\p_v$ of good reduction and for
$k\in\Z_{\geq 1}$, the local Euler factor gives
\begin{equation}\label{eq:lambda-sym2_prime_power}
    \lambda_{\Sym^2}(\p_v^k)
    =
    \alpha_v^{2k}+\alpha_v^k\beta_v^k+\beta_v^{2k}
    =
    \alpha_v^{2k}+q_v^k+\beta_v^{2k}.
\end{equation}
%The coefficients at the finitely many places of bad reduction are likewise determined by the corresponding local Euler factors. 
In particular, when $k=1$,
\begin{equation}\label{eq:lambda-sym2}
    \lambda_{\Sym^2}(\p_v)
    =
    \alpha_v^{2}+\alpha_v\beta_v+\beta_v^{2}
    =
    \alpha_v^2+q_v+\beta_v^2
    =
    a_v(E)^2-q_v,
\end{equation}
which is exactly the numerator appearing in the sum
\eqref{eq:S2-def}.

\begin{comment}
Let $\lambda_{\Sym^2}(\a)$ denote the Dirichlet coefficients of $L(\Sym^2(E/K),s)$, so that for $\RE(s)>2$
\begin{equation}
    L(\Sym^2(E/K), s)=\sum_{\a\subseteq \cO_K} \frac{\lambda_{\Sym^2}(\a)}{N_{K/\Q}(\a)^s}.
\end{equation}
For a non-zero prime ideal $\p_v$ of good reduction and $k\in \Z_{\geq 1}$,
\begin{equation}\label{eq:lambda-sym2_prime_power}
    \lambda_{\Sym^2}(\p_v^k)=\alpha_v^{2k} + \alpha_v^k\beta_v^k + \beta_v^{2k}=\alpha_v^{2k}+q_v^k + \beta_v^{2k}.
\end{equation}
In particular, when $k=1$, 
\begin{equation}\label{eq:lambda-sym2}
    \lambda_{\Sym^2}(\p_v)=\alpha_v^{2} + \alpha_v\beta_v + \beta_v^{2}=\alpha_v^2+q_v+\beta_v^2=a_v(E)^2-q_v,
\end{equation}
which is exactly the numerator appearing in the sum \eqref{eq:S2-def}. The
logarithmic derivative therefore satisfies
\begin{equation}\label{eq:L'/L_Sym^2}
    -\frac{L'}{L}(\Sym^2(E/K),s)
    = \sum_{\a\subseteq \cO_K} \frac{\lambda_{\Sym^2}(\a)\,\Lambda_K(\a)}{N_{K/\Q}(\a)^{s}}
    = \sum_{k=1}^\infty \sum_{v\in \Val_0(K)} \frac{\lambda_{\Sym^2}(\p_v^k)\log(q_v)}{q_v^{ks}}.
\end{equation}
\end{comment}

Define
\begin{equation}\label{eq:deltaE-def}
\delta_E\defeq
\begin{cases}
1 & \text{if $E$ has CM defined over $K$},\\
0 & \text{otherwise}.
\end{cases}
\end{equation}

\begin{proposition}\label{prop:S2_estimate}
    Assume that $E$ is modular and that the Generalized Riemann Hypothesis
    holds for $L(\Sym^2(E/K), s)$.
    Then
    \begin{equation}\label{eq:S2-estimate}
        S_2(X) = \delta_E\,\frac{\phi(0)}{4}\log(X) + O_{\phi,K}\big(\log\log(A_{E/K})\big).
    \end{equation}
\end{proposition}

The proof splits into cases, depending on whether $E$ has potential complex multiplication (PCM) or not. We treat the non-PCM case first, where $L(\Sym^2(E/K),s)$ is itself a nice cuspidal $L$-function, and
then reduce the PCM case to Hecke $L$-functions via a factorization of
$L(\Sym^2(E/K),s)$.

\medskip
\noindent\textbf{The non-PCM case.}
Suppose $E$ does not have PCM. By Serre's Open Image Theorem
\cite[IV-11]{Ser68}, the adelic Galois representation
\begin{equation}
    \rho_E: \Gal(\overline{K}/K) \to \GL_2(\widehat{\Z})
\end{equation}
has open image; in particular $\rho_E$ is non-dihedral. It then follows from
the work of Gelbart and Jacquet \cite[\S 3 and (9.3) Theorem]{GJ78} that,
since $E$ is modular and non-PCM, the symmetric square
$L(\Sym^2(E/K),s)$ is the $L$-function of a cuspidal automorphic representation of $\GL_3$, and is therefore entire. 
The Hasse bound $|\alpha_v|=|\beta_v|=\sqrt{q_v}$ shows that this representation satisfies the Ramanujan--Petersson bound at every unramified place. Hence, by \cite[Theorem 5.17]{IK04}, under GRH for $L(\Sym^2(E/K),s)$, for $\frac{3}{2}\le \RE(s)\le 2$, we have
\begin{equation}\label{eq:L'/L_Sym^2_bound}
    \left|\frac{L'}{L}(\Sym^2(E/K), s)\right| \ll_K \log\log(A_{E/K}),
\end{equation}
where we have used that the analytic conductor of
$\Sym^2(E/K)$ is polynomially bounded in $A_{E/K}$.

\medskip
\noindent\textbf{The PCM case.}
Now suppose $E$ has PCM. We associate to $E$ a quadratic $K$-algebra $M$ as
follows: if $E$ has CM defined over $K$, set $M\defeq K\times K$; otherwise let
$M$ be the (unique) quadratic extension of $K$ over which $E$ acquires CM. Let
$\eta_{M/K}$ be the quadratic Hecke character of $K$ attached to $M/K$ (the
trivial character when $M=K\times K$), and let $\psi_{E/K}$ denote the
associated Gr\"ossencharacter. By a theorem of Deuring (see, e.g.,
\cite[\S 2 Theorem 10.5]{Sil94}),
\begin{equation}\label{eq:deuring}
    L(E/K, s) \cong L(\Ind_M^K(\psi_{E/K}), s) \cong
    \begin{cases}
        L(\psi_{E/M}, s)\, L(\overline{\psi_{E/M}}, s) & \text{ if } M\cong K\times K, \\
        L(\psi_{E/M}, s) & \text{ otherwise;}
    \end{cases}
\end{equation}
in particular $\rho_E$ is dihedral when $E$ has PCM. In this case $L(\Sym^2(E/K),s)$ factorizes as a product of Hecke $L$-functions.

\begin{lemma}\label{lem:PCM-factorization}
    Let $E$ be an elliptic curve over $K$ with PCM, and let $M$ and
    $\eta_{M/K}$ be as above. Then there is a finite product of local Euler
    factors $C_E(s)$ such that
    \begin{equation}\label{eq:PCM-factorization}
        L(\Sym^2(E/K), s) = C_E(s)\, L(\eta_{M/K}, s-1)\, L(\Ind_M^K(\psi_{E/K}^2), s).
    \end{equation}
\end{lemma}

The proof of Lemma~\ref{lem:PCM-factorization} will use the following
representation-theoretic identity.

\begin{lemma}
\label{lem:induced-sym-square}
Let $G$ be a group, let $H\triangleleft G$ be a subgroup of index $2$, let
\begin{equation}
    \eta:G\to\{\pm 1\}
\end{equation}
be the quadratic character with kernel $H$, let $F$ be a field of characteristic not equal to $2$, and let 
\begin{equation}
    \psi:H\to F^\times
\end{equation} 
be a character. 
Then
\begin{equation}
\Sym^2 \Ind_H^G(\psi)
\simeq
\Ind_H^G(\psi^2)\oplus \left(\eta\otimes\det \left(\Ind_H^G(\psi)\right)\right).
\end{equation}
\end{lemma}

\begin{proof}
Choose an element $\tau\in G\setminus H$.  Since $H$ has index $2$, it is normal.  Define the character
\begin{salign}
    \psi^\tau: H &\to F^\times\\
        h &\mapsto \psi(\tau^{-1}h\tau).
\end{salign}
Define $V\defeq \Ind_H^G(\psi)$. Then $\dim_F V=2$ and
\begin{equation}
    V|_H\simeq \psi\oplus\psi^\tau.
\end{equation} 
Choose a basis $e_1,e_2$ of $V$ such that, for $h\in H$,
\begin{equation}
    h e_1=\psi(h)e_1,\qquad \text{ and }\qquad h e_2=\psi^\tau(h)e_2.
\end{equation}
Since $\tau$ swaps the two $H$-eigenlines, we may rescale so that
\begin{equation}\label{eq:tau_action}
    \tau\cdot e_1=e_2,\qquad \tau\cdot e_2=a\, e_1
\end{equation}
for some $a\in F^\times$. Thus, in the basis $e_1$, $e_2$, 
\begin{equation}
    V(\tau)=\left(\begin{matrix}0&a\\1&0\end{matrix}\right).
\end{equation}

Now consider $\Sym^2 V$ with the basis
\begin{equation}
e_1^2,\qquad e_1e_2,\qquad e_2^2.
\end{equation}
The subspace
$\langle e_1^2,e_2^2\rangle$
is stable under $G$.  Indeed, for $h\in H$,
\begin{equation}
h e_1^2=h(e_1) h(e_1) = \psi(h)^2 e_1^2,\qquad
h e_2^2= h(e_2)h(e_2)=(\psi^\tau(h))^2e_2^2,
\end{equation}
and by \eqref{eq:tau_action} $\tau$ interchanges the lines $\langle e_1^2\rangle$ and $\langle e_2^2\rangle$ up to scaling. 
Hence $\langle e_1^2,e_2^2\rangle\simeq \Ind_H^G(\psi^2)$.

We now consider the remaining line, $\langle e_1e_2\rangle$.  For $h\in H$,
\begin{equation}
h(e_1e_2)= h(e_1) h(e_2) =\psi(h)\psi^\tau(h)e_1e_2
\end{equation}
and 
\begin{equation}
    \tau(e_1e_2)=(\tau e_1)(\tau e_2)=e_2(ae_1)=a e_1e_2.
\end{equation}
Therefore $\langle e_1e_2\rangle$ is $G$-stable. 
Additionally,
\begin{equation}
\det (V)(h)=\psi(h)\psi^\tau(h) \text{ and } 
\det (V)(\tau)
=
\det
\begin{pmatrix}
0&a\\
1&0
\end{pmatrix}
=-a,
\end{equation}
and $\eta(h)=1$ and $\eta(\tau)=-1$.
Therefore $H$ acts on $\langle e_1e_2\rangle$ through
$(\eta\otimes\det V)|_H$,
and $G$ acts on $\langle e_1 e_2\rangle$ by the character $\eta\otimes \det(V)$. 
 Thus
\begin{equation}
\Sym^2 V
=
\langle e_1^2,e_2^2\rangle\oplus \langle e_1e_2\rangle
\simeq
\Ind_H^G(\psi^2)\oplus(\eta\otimes\det V).
\end{equation}
\end{proof}

We now apply Lemma \ref{lem:induced-sym-square} to elliptic curves with PCM to prove Lemma \ref{lem:PCM-factorization}.

\begin{proof}[Proof of Lemma~\ref{lem:PCM-factorization}]
    Let $\ell$ be a prime, let
    \begin{equation}
        \rho_{E,\ell}: G_K \to \GL_2(\Q_\ell)
    \end{equation}
    be the $\ell$-adic Galois representation attached to $E$, and let
    \begin{equation}
        \chi_\ell: G_K \to \Z_\ell^\times.
    \end{equation}
    be the $\ell$-adic cyclotomic character.
    By properties of the Weil pairing (see, e.g., \cite[\S III.8]{Sil09}), $\det\rho_{E,\ell}=\chi_\ell$. Since $E$
    has PCM, $\rho_{E,\ell}\simeq \Ind_{M}^{K}(\psi_{E/K,\ell})$ (equation \ref{eq:deuring}), so applying Lemma~\ref{lem:induced-sym-square} with $G=G_K$,
    $H=G_M$, and $\psi=\psi_{E/K,\ell}$ gives
    \begin{equation}\label{eq:Sym2_rho_identity}
        \Sym^2(\rho_{E,\ell}) \cong
        \Ind_M^K(\psi_{E/K,\ell}^2)\;\oplus\;\big(\eta_{M/K}\otimes \chi_\ell\big).
    \end{equation}
    Let $v\in \Val_0(K)$ be a place at which $E$ has good reduction, $M/K$ is
    unramified, and all characters in \eqref{eq:Sym2_rho_identity} are
    unramified; note that all but finitely many places satisfy these conditions. 
    Let $\Frob_v\in G_K$ be a Frobenius element at $v$. Since $M/K$ is unramified at $v$, $\chi_\ell(\Frob_v)=q_v$, so that the Euler factor of
    $\eta_{M/K}\otimes\chi_\ell$ at $v$ is
    \begin{equation}
        \left(1 - \frac{\eta_{M/K}(v)\,q_v}{q_v^{s}}\right)^{-1}
        =\left(1 - \frac{\eta_{M/K}(v)}{q_v^{s-1}}\right)^{-1},
    \end{equation}
    which is precisely the Euler factor of $L(\eta_{M/K}, s-1)$ at $v$.
    Comparing Euler factors in \eqref{eq:Sym2_rho_identity} at all such $v$,
    the $L$-functions of $\Sym^2(E/K)$ and of
    $L(\eta_{M/K},s-1)\,L(\Ind_M^K(\psi_{E/K}^2),s)$ agree away from a finite
    set of places; the discrepancy at the remaining places is a finite product of Euler factors, which we may denote by $C_E(s)$. This gives the factorization \eqref{eq:PCM-factorization}.
\end{proof}

\begin{comment}
    We record, for use throughout this subsection, that the sum \eqref{eq:S2-def} defining $S_2(X)$ converges absolutely under the standing hypotheses on $\phi$ from the beginning of Section~\ref{sec:L-functions} (those preceding Proposition~\ref{prop:modified-expicit-formula}). Indeed, the Hasse bound gives $|\lambda_{\Sym^2}(\p_v)|=|a_v(E)^2-q_v|\le 3q_v$, so the summand of \eqref{eq:S2-def} is $O\!\big(\tfrac{\log(q_v)}{q_v}\,|\widehat\phi(\tfrac{2\log q_v}{\log X})|\big)$, and $\sum_v \tfrac{\log(q_v)}{q_v}|\widehat\phi(\tfrac{2\log q_v}{\log X})|<\infty$ by the exponential decay of $\widehat\phi$ recorded in \eqref{eq:phihat_decay}. We do not assume that $\widehat\phi$ is compactly supported.
\end{comment} 

Our approach to proving Proposition \ref{prop:S2_estimate} is to apply Abel summation to the sum $S_2(X)$. 
Doing so, we are led to study the partial sums
\begin{equation}\label{eq:Theta-def}
    \Theta(x)\defeq \sum_{\substack{v\in\Val_0(K)\\ E\text{ good at }v\\ q_v\le x}} \frac{\lambda_{\Sym^2}(\p_v)\log(q_v)}{q_v^{2}},
\end{equation}
whose asymptotics are governed by the behavior of $L(\Sym^2(E/K),s)$ at $s=2$.
Define
\begin{equation}
    \delta_E'\defeq -\ord_{s=2}L(\Sym^2(E/K),s).
\end{equation}

\begin{lemma}\label{lem:PNT-sym2}
Assume that $E$ is modular and that GRH holds for $L(\Sym^2(E/K),s)$. Then, for $x\ge2$,
\begin{equation}\label{eq:PNT-sym2}
    \Theta(x)=\delta_E'\,\log(x)+O_{K}\!\left(\log\log(A_{E/K})\right).
\end{equation}
\end{lemma}

\begin{proof}
We first isolate the arithmetic content of $\Theta(x)$ inside the logarithmic
derivative of $L(\Sym^2(E/K),s)$. By \eqref{eq:lambda-sym2}, the quantity
$\lambda_{\Sym^2}(\p_v)\log(q_v)$ is exactly the $k=1$, good-reduction part
of the Dirichlet coefficient of
$-\frac{L'}{L}(\Sym^2(E/K),s)$ in \eqref{eq:L'/L_Sym^2}.

Define
\begin{equation}\label{eq:Psi-sym2-def}
    \Psi_{\Sym^2}(x)
    \defeq
    \sum_{\substack{\a\subseteq\cO_K\\N_{K/\Q}(\a)\leq x}}
    \lambda_{\Sym^2}(\a)\Lambda_K(\a).
\end{equation}
The explicit formula for $-\frac{L'}{L}(\Sym^2(E/K),s)$, together with GRH
for $L(\Sym^2(E/K),s)$, gives
\begin{equation}\label{eq:Psi-sym2-estimate}
    \Psi_{\Sym^2}(x)
    =
    \delta_E'\frac{x^2}{2}
    +
    O_K\!\left(x^{3/2}\log^2(A_{E/K}x)\right).
\end{equation}
Indeed, the pole at $s=2$ contributes the main term
$\delta_E'x^2/2$, while GRH places all non-trivial zeros on the line
$\Re(s)=3/2$.

We now remove from the asymptotic \eqref{eq:Psi-sym2-estimate} the contributions from
prime powers $\p_v^k$ with $k\geq 2$ and from the finitely many places of
bad reduction. For $k\geq 2$, equation
\eqref{eq:lambda-sym2_prime_power} and the Hasse bound give
\begin{equation}
    |\lambda_{\Sym^2}(\p_v^k)|
    \leq
    |\alpha_v|^{2k}+q_v^k+|\beta_v|^{2k}
    \leq 3q_v^k.
\end{equation}
It follows that
\begin{equation}
    \sum_{\substack{v\in\Val_0(K)\\ k\geq 2\\q_v^k\leq x}}
    \lambda_{\Sym^2}(\p_v^k)\log(q_v)
    \ll_K x^{3/2}.
\end{equation}
The contribution from the bad places is also absorbed by the error term in
\eqref{eq:Psi-sym2-estimate}. Thus, if
\begin{equation}\label{eq:psi-sym2-prime-def}
    \psi_{\Sym^2}(x)
    \defeq
    \sum_{\substack{v\in\Val_0(K)\\E\text{ good at }v\\q_v\leq x}}
    \lambda_{\Sym^2}(\p_v)\log(q_v),
\end{equation}
then
\begin{equation}\label{eq:psi-sym2-prime-estimate}
    \psi_{\Sym^2}(x)
    =
    \delta_E'\frac{x^2}{2}
    +
    O_K\!\left(x^{3/2}\log^2(A_{E/K}x)\right).
\end{equation}

We next define the constant that will occur in the asymptotic for
$\Theta(x)$. Writing
\begin{equation}\label{eq:D-def}
    D(w)
    \defeq
    \sum_{\substack{v\in\Val_0(K)\\E\text{ good at }v}}
    \frac{\lambda_{\Sym^2}(\p_v)\log(q_v)}{q_v^{2}}\,q_v^{-w},
\end{equation}
there exists a function $H(w)$ that is holomorphic and bounded on
$\Re(w)\geq-\tfrac12+\epsilon$, such that
\begin{equation}\label{eq:D-vs-LL}
    D(w)
    =
    -\frac{L'}{L}\big(\Sym^2(E/K),w+2\big)+H(w).
\end{equation}
Indeed, the difference between the two sides is accounted for by the
prime-power terms with $k\geq 2$ and the finitely many bad places.

The function $D(w)$ has a simple pole at $w=0$ with residue $\delta_E'$.
We define $c_E$ to be the constant term in its Laurent expansion at
$w=0$:
\begin{equation}\label{eq:cE-def}
    D(w)=\frac{\delta_E'}{w}+c_E+O(w)
    \qquad\text{as }w\to 0.
\end{equation}
By \eqref{eq:D-vs-LL}, we have
\begin{equation}
    c_E
    =
    \lim_{s\to 2}
    \left(
        -\frac{L'}{L}(\Sym^2(E/K),s)
        -
        \frac{\delta_E'}{s-2}
    \right)
    +
    H(0).
\end{equation}
By \cite[Theorem 5.17]{IK04}, under GRH the regular part of
$-\frac{L'}{L}(\Sym^2(E/K),s)$ at $s=2$ is
$O_K(\log\log(A_{E/K}))$. Since $H(0)=O_K(1)$, it follows that
\begin{equation}\label{eq:cE-bound}
    c_E=O_K\!\left(\log\log(A_{E/K})\right).
\end{equation}

We now apply Abel summation to the sum \eqref{eq:Theta-def}. For $x\geq 2$, we have
\begin{equation}\label{eq:Theta-Abel}
    \Theta(x)
    =
    \frac{\psi_{\Sym^2}(x)}{x^2}
    +
    2\int_2^x\frac{\psi_{\Sym^2}(t)}{t^3}\,dt
    +
    O_K(1).
\end{equation}
Substituting \eqref{eq:psi-sym2-prime-estimate} into
\eqref{eq:Theta-Abel}, we obtain
\begin{equation}\label{eq:Theta-preliminary}
    \Theta(x)
    =
    \delta_E'\log(x)
    +
    c_E
    +
    O_K\!\left(x^{-1/2}\log^2(A_{E/K}x)\right).
\end{equation}
Here the constant term is $c_E$ as defined in \eqref{eq:cE-def}; this
follows by applying partial summation to the Dirichlet series $D(w)$ and
comparing the constant terms at $w=0$.

Finally, combining \eqref{eq:Theta-preliminary} with
\eqref{eq:cE-bound} gives
\begin{equation}
    \Theta(x)
    =
    \delta_E'\log(x)
    +
    O_K\!\left(\log\log(A_{E/K})\right).
\end{equation}
\end{proof}

%We remark that only the value of \eqref{eq:L'/L_Sym^2_bound} on the approach to $s=2$ is used; the estimate \eqref{eq:L'/L_Sym^2_bound} is stated on the whole segment $\tfrac32\le\RE(s)\le2$ (on which it holds uniformly) merely because that is the natural region between the central line and the edge of absolute convergence, but the sub-interval near $s=2$ would suffice for our purposes.

\begin{proof}[Proof of Proposition~\ref{prop:S2_estimate}]
    By the expression \eqref{eq:lambda-sym2}, $\lambda_{\Sym^2}(\p_v)=a_v(E)^2-q_v$. 
    By the bound \eqref{eq:finite_places_ignorable} the contribution to $S_2$ of the finitely many places of bad reduction is negligible, and the sum \eqref{eq:S2-def} becomes
    \begin{equation}\label{eq:S2-as-abel}
        S_2(X)=\sum_{\substack{v\in\Val_0(K)\\ E\text{ good at }v}}\frac{\lambda_{\Sym^2}(\p_v)\log q_v}{q_v^{2}}\,\widehat\phi\!\left(\frac{2\log q_v}{\log(X)}\right)+O_{\phi,K}(1),
    \end{equation}
    which is exactly the summand of \eqref{eq:Theta-def} weighted by the bounded factor $\widehat\phi\big(\tfrac{2\log q_v}{\log(X)}\big)$. We evaluate \eqref{eq:S2-as-abel} by Abel summation against the partial sums $\Theta$ of \eqref{eq:Theta-def}. Let $f(t)\defeq\widehat\phi\big(\tfrac{2\log t}{\log(X)}\big)$, so that $f'(t)=\tfrac{2}{t\log(X)}\widehat\phi{}'\big(\tfrac{2\log t}{\log(X)}\big)$. Abel summation gives
    \begin{equation}\label{eq:S2-abel-int}
        S_2(X)=\int_{2}^{\infty} f(t)\,d\Theta(t)+O_{\phi,K}(1)
        =-\int_{2}^{\infty}\Theta(t)\,f'(t)\,dt+O_{\phi,K}(1).
    \end{equation}
    The boundary term of \eqref{eq:S2-abel-int} at $t=\infty$ vanishing because, by Lemma \ref{lem:PNT-sym2}, 
    \begin{equation}
        \Theta(t)=O_K(\log t+\log\log A_{E/K})
    \end{equation} 
    grows logarithmically while $f(t)=\widehat\phi\big(\tfrac{2\log t}{\log X}\big)$ decays faster than any power of $t$ by the assumption \eqref{eq:phihat_decay} on $\phi$. 
    The boundary term at $t=2$ is $O_K(1)$ since $\Theta(2)=O_K(1)$.
    
    Substituting the asymptotic \eqref{eq:PNT-sym2} for $\Theta$, equation \eqref{eq:S2-abel-int} becomes
    \begin{equation}\label{eq:S2-abel-sum_2}
        S_2(X) =  -\delta_E'\int_2^\infty \log(t)\,f'(t)\,dt 
        +
        O_K\!\left(\log\log(A_{E/K})\int_2^{\infty}|f'(t)|\,dt\right) + O_{K,\phi}(1).
    \end{equation}
    By equation \eqref{eq:f'_integral} and the bound \eqref{eq:phihat_decay},
    \begin{equation}
        \int_2^\infty|f'(t)|\,dt=O_\phi(1),
    \end{equation}
    and from this it follows that the error term in \eqref{eq:S2-abel-sum_2} simplifies to $O_K(\log\log A_{E/K})$.

    For the main term of \eqref{eq:S2-abel-sum_2} we integrate by parts, using $\tfrac{d}{dt}\log t=\tfrac1t$ and the vanishing of the boundary terms:
    \begin{equation}\label{eq:S2-main-computation}
        -\delta_E'\int_2^\infty \log(t)\,f'(t)\,dt
        =\delta_E'\int_2^\infty \frac{f(t)}{t}\,dt
        =\delta_E'\int_2^\infty \frac{1}{t}\,\widehat\phi\!\left(\frac{2\log t}{\log(X)}\right)dt.
    \end{equation}
    As in the evaluation \eqref{eq:log(X)phi(0)/4}, making the substitution $u=\frac{2\log t}{\log(X)}$, \eqref{eq:S2-main-computation} becomes
    \begin{equation}
        \delta_E'\frac{\log(X)}{2}\int_0^\infty\widehat\phi(u)\,du+O_{\phi}(1)
        =\delta_E'\frac{\phi(0)}{4}\log(X)+O_\phi(1).
    \end{equation}
    We therefore have
    \begin{equation}\label{eq:S2-explicit}
        S_2(X)=\delta_E'\,\frac{\phi(0)}{4}\log(X)+O_{\phi,K}\!\big(\log\log(A_{E/K})\big).
    \end{equation}
    It remains to determine $\delta_E'=-\ord_{s=2}L(\Sym^2(E/K),s)$.

    \emph{Non-PCM case.} As mentioned above, $L(\Sym^2(E/K),s)$ is entire and cuspidal, and therefore has neither a zero nor a pole at $s=2$. Hence $\delta_E'=0$, and \eqref{eq:S2-explicit} gives \eqref{eq:S2-estimate} with $\delta_E=0$.

    \emph{PCM case, CM not defined over $K$.} In this case $\eta_{M/K}$ is a non-trivial quadratic Hecke character, so $L(\eta_{M/K}, s-1)$ is entire, and therefore holomorphic and non-vanishing at $s=2$. 
    Likewise,  $L(\Ind_M^K(\psi_{E/K}^2),s)$ is entire (as $\psi_{E/K}^2$ is a non-trivial Gr\"ossencharacter). Thus $L(\Sym^2(E/K),s)$ is holomorphic and non-vanishing at $s=2$, so $\delta_E'=0$ and \eqref{eq:S2-explicit} gives \eqref{eq:S2-estimate} with $\delta_E=0$.

    \emph{PCM case, CM defined over $K$.} In this case $M=K\times K$ and $\eta_{M/K}$ is trivial, so that $L(\eta_{M/K},s-1)=\zeta_K(s-1)$, which has a simple pole at $s=2$. The remaining factor $L(\Ind_M^K(\psi_{E/K}^2),s)$ is entire and non-vanishing at $s=2$, so $L(\Sym^2(E/K),s)$ has a simple pole at $s=2$ and $\delta_E'=1$. Hence \eqref{eq:S2-explicit} gives \eqref{eq:S2-estimate} with $\delta_E=1$.
\end{proof}

\begin{remark}
    The $k=2$ contribution can alternatively be handled unconditionally, by
    averaging over the family $\cE_K(B)$ (see, e.g.,
    \cite{BZ08, CJ23, Phi25}). We have chosen instead to assume GRH for $L(\Sym^2(E/K), s)$ (as in \cite{ILS00, You06}), as this simplifies the computation of the moments and since we are already assuming GRH for Hasse--Weil $L$-functions.
\end{remark}

%---------------------------------------------------------------------
\subsection{An expression for the \texorpdfstring{$1$}{1}-level density}
%---------------------------------------------------------------------

Before assembling the pieces, we record a standard bound on the conductor.
Let $\fD_{E/K}$ denote the minimal discriminant of $E$. If $E$ has naive
height $\Ht(E)\le X$, then $N_{K/\Q}(\fD_{E/K})\ll_K X$, and therefore
\begin{equation}
\log\!\big(N_{K/\Q}(\fD_{E/K})\big)\leq \log(X)+O(1).
\end{equation}
Since the conductor $\f_{E/K}$ divides $\fD_{E/K}$, we have
$N_{K/\Q}(\f_{E/K})\leq N_{K/\Q}(\fD_{E/K})$, and thus
\begin{equation}\label{eq:conductor_bound}
     \log(A_{E/K})=\log(N_{K/\Q}(\f_{E/K})\Delta_K^2)\leq \log(X)+O(1),
\end{equation}
so that
\begin{equation}\label{eq:loglog(A)/log(x)->o(1)}
\frac{\log\log(A_{E/K})}{\log(X)} \ll \frac{\log\log(X)}{\log(X)}=o(1).
\end{equation}

The $k=1$ term of \eqref{eq:sum_ideals_to_sum_prime_powers} is the
arithmetically significant one. Restricting to places above rational primes $>3$ (which is justified by the bound \eqref{eq:finite_places_ignorable}), we
define
\begin{equation}\label{eq:U1-def}
    U_{1}(\phi,E, X)\defeq \sum_{\substack{v\in \Val_0(K) \\ 2\nmid q_v,\ 3\nmid q_v}} \frac{\widehat{a_E}(\p_v)\log(q_v)}{q_v} \cdot \widehat{\phi}\left( \frac{\log(q_v)}{\log(X)}\right).
\end{equation}

\begin{proposition}\label{prop:1-level_density_expression}
Maintain the notation and hypotheses of
Proposition~\ref{prop:modified-expicit-formula}, and assume GRH for $L(\Sym^2(E/K),s)$. Let $E\in\cE_K(X)$, and let $\delta_E$ be as in
\eqref{eq:deltaE-def}. Then
\begin{equation}\label{eq:1-level_density_expression}
    D_1(E/K,\phi)
    = \widehat{\phi}(0) + \frac{\phi(0)}{2}\,(1-\delta_E) -\frac{2}{\log(X)}\, U_1(\phi,E ,X)
    + O_{\phi,K}\!\left(\frac{\log\log(X)}{\log (X)}\right),
\end{equation}
where the implied constant depends only on $\phi$ and $K$.
In particular, since $\delta_E\in\{0,1\}$,
\begin{equation}\label{eq:1-level_density_bound}
    D_1(E/K,\phi)
    \leq \widehat{\phi}(0) + \frac{\phi(0)}{2} -\frac{2}{\log(X)}\, U_1(\phi,E ,X)
    + O_{\phi,K}\!\left(\frac{\log\log(X)}{\log (X)}\right).
\end{equation}
\end{proposition}

\begin{proof}
    By the explicit formula \eqref{eq:modified-explicit-formula},
    the decomposition \eqref{eq:sum_ideals_to_sum_prime_powers}, and
    the conductor bound \eqref{eq:conductor_bound}, we have
    \begin{align}
        D_1(E/K,\phi)
        &= \widehat{\phi}(0) + O_{\phi,K}\!\left(\tfrac{1}{\log(X)}\right)
        - \frac{2}{\log(X)}\Bigg[
        \underbrace{\sum_{v}\frac{\widehat{a_E}(\p_v)\log q_v}{q_v}\widehat\phi\!\left(\tfrac{\log q_v}{\log(X)}\right)}_{k=1}
        \notag\\
        &\qquad
        + \underbrace{\sum_{v}\frac{\widehat{a_E}(\p_v^2)\log q_v}{q_v^2}\widehat\phi\!\left(\tfrac{2\log q_v}{\log(X)}\right)}_{k=2}
        + \underbrace{\sum_{v}\sum_{k\ge 3}\frac{\widehat{a_E}(\p_v^k)\log q_v}{q_v^k}\widehat\phi\!\left(\tfrac{k\log q_v}{\log(X)}\right)}_{k\ge 3}
        \Bigg].
    \end{align}
    By Lemma~\ref{lem:k>=3}, the $k\geq 3$ term is $O(1)$. For the $k=1$ term,
    including the finitely many places above $2$ and $3$ changes the sum by
    $O(1)$ by the bound \eqref{eq:finite_places_ignorable}, so it equals $U_1(\phi,E,X)+O(1)$. For the $k=2$ term, equation \eqref{eq:k=2_sum} together with Lemma~\ref{lem:k=2_1} and Proposition~\ref{prop:S2_estimate} gives
    \begin{salign}
        \sum_{v}\frac{\widehat{a_E}(\p_v^2)\log q_v}{q_v^2}\widehat\phi\!\left(\tfrac{2\log q_v}{\log(X)}\right)
        &= S_2(X) - \frac{\phi(0)}{4}\log(X) + O_{\phi,K}(1) \\
        &= (\delta_E-1)\,\frac{\phi(0)}{4}\log(X) + O_{\phi,K}(\log\log A_{E/K}).
    \end{salign}
    Substituting these evaluations, multiplying the $k=2$ term by
    $-\tfrac{2}{\log(X)}$, and using the bound \eqref{eq:loglog(A)/log(x)->o(1)} to absorb the error, we obtain \eqref{eq:1-level_density_expression}.
\end{proof}

Taking $m$-th powers of the inequality \eqref{eq:1-level_density_bound} and expanding via the
binomial theorem, we have
\begin{salign}\label{eq:m-th_power_of_1-level_density}
    D_1(E/K,\phi)^m
    & \leq \left(\widehat{\phi}(0) + \frac{\phi(0)}{2} - \frac{2}{\log(X)} U_1(\phi,E, X) \right)^m \\
    &\quad  + \sum_{k=1}^m \binom{m}{k} \left(\widehat{\phi}(0) + \frac{\phi(0)}{2} - \frac{2}{\log(X)} U_1(\phi,E, X)\right)^{m-k} O\!\left(\frac{\log\log(X)}{\log(X)}\right)^k.
\end{salign}
The dominant error contribution comes from the $k=1$ term, namely
\begin{salign}\label{eq:error_term_m-power_of_1-level_density}
    & m \left(\widehat{\phi}(0) + \frac{\phi(0)}{2} - \frac{2}{\log(X)} U_1(\phi,E, X)\right)^{m-1} O\!\left(\frac{\log\log(X)}{\log(X)}\right) \\
    &\quad = O_m\!\left(\frac{\log\log(X)}{\log(X)} \sum_{j=0}^{m-1} \binom{m-1}{j} \left(\widehat{\phi}(0) + \frac{\phi(0)}{2}\right)^{m-1-j} \frac{(-2)^j}{\log(X)^j}\, U_1(\phi,E,X)^j\right).
\end{salign}
Expanding the main term of \eqref{eq:m-th_power_of_1-level_density}, we have
\begin{salign}\label{eq:main_term_m-power_of_1-level_density}
    & \left(\widehat{\phi}(0) + \frac{\phi(0)}{2} -\frac{2}{\log(X)}U_1(\phi,E, X) \right)^m
 \\
 &\hspace{2cm} = \sum_{j=0}^m \binom{m}{j} \left(\widehat{\phi}(0) + \frac{\phi(0)}{2}\right)^{m-j} \frac{(-2)^j}{\log(X)^j} U_1(\phi,E, X)^j.
\end{salign}
Substituting equations \eqref{eq:error_term_m-power_of_1-level_density} and
\eqref{eq:main_term_m-power_of_1-level_density} into
\eqref{eq:m-th_power_of_1-level_density}, we find that $D_1(E/K,\phi)^m$ is
bounded above by
\begin{salign}
    &  \sum_{j=0}^m \binom{m}{j} \left(\widehat{\phi}(0) + \frac{\phi(0)}{2}\right)^{m-j} \frac{(-2)^j}{\log(X)^j}\, U_1(\phi,E, X)^j \\
    &\ + O_m\!\left(\frac{\log\log(X)}{\log(X)} \sum_{j=0}^{m-1} \binom{m-1}{j} \left(\widehat{\phi}(0) + \frac{\phi(0)}{2}\right)^{m-1-j} \frac{(-2)^j}{\log(X)^j}\, U_1(\phi,E,X)^j\right).
\end{salign}

We now average over isomorphism classes of elliptic curves over $K$. For
$j\in\Z_{\geq 0}$ and $B>0$, define
\begin{salign}\label{eq:S_1_sum}
S_1(j, \phi ,B)
&\defeq \frac{1}{\#\cE_{K}(B)} \sum_{E\in \cE_{K}(B)} U_1(\phi,E,B)^j\\
&=\frac{1}{\#\cE_{K}(B)} \sum_{\substack{v_1,\dots,v_j\in \Val_0(K) \\ 2\nmid q_{v_i},\ 3\nmid q_{v_i}}} \prod_{i=1}^j \frac{\log(q_{v_i})}{q_{v_i}} \cdot \widehat{\phi}\left( \frac{\log(q_{v_i})}{\log(B)}\right)
\sum_{E\in\cE_{K}(B)}  \prod_{i=1}^j \widehat{a_E}(\p_{v_i}),
\end{salign}
where the $v_i$ need not be distinct. Then the $m$-th moment of the $1$-level
density,
\begin{equation}
    \frac{1}{\#\cE_{K}(B)} \sum_{E\in \cE_{K}(B)} D_1(E/K, \phi)^m,
\end{equation}
is bounded above by
\begin{salign}\label{eq:1-level_moments_explicit_formula}
    &  \sum_{j=0}^m \binom{m}{j} \left(\widehat{\phi}(0) + \frac{\phi(0)}{2}\right)^{m-j} \frac{(-2)^j}{\log(B)^j}\, S_1(j,\phi, B) \\
    &\ + O_m\!\left(\frac{\log\log(B)}{\log(B)} \sum_{j=0}^{m-1} \binom{m-1}{j} \left(\widehat{\phi}(0) + \frac{\phi(0)}{2}\right)^{m-1-j} \frac{(-2)^j}{\log(B)^j}\, S_1(j,\phi,B)\right).
\end{salign}
To further analyze moments of the 1-level density of elliptic curves, we approximate the $S_1(j,\phi,B)$. A key ingredient in these approximations will be estimating the number of elliptic curves with prescribed $\widehat{a_E}(\p_{v_i})$, which we discuss in the next section.

\section{Counting elliptic curves with local conditions}\label{sec:Counting_ECs}

In this section we collect results related to counting elliptic curves with prescribed $\widehat{a_E}(\p)$ from previous work of the first author \cite{Phi25} relevant for approximating the sums $S_1(j,\phi,B)$ defined in \eqref{eq:S_1_sum}.

We first state results for counting elliptic curves with a prescribed local condition.

\begin{proposition}[{\cite[Theorem 1.1.2]{Phi25}}]\label{prop:LocalConditions}
Let $K$ be a degree $d$ number field, and let $\p$ be a prime ideal of $\cO_K$ of norm $q$ such that $2\nmid q$ and $3\nmid q$. Let $\sL$ be one of the local conditions listed in Table \ref{tab:LocalConditions}. Then there exists an explicit constant $\kappa(K)$, depending only on $K$, such that the number of elliptic curves over $K$ with naive height less than $B$, and which satisfy the local condition $\sL$ at $\p$, is
\begin{equation}
\kappa(K)\kappa_\sL B^{5/6}+O\left(\epsilon_\sL B^{\frac{5}{6}-\frac{1}{3d}}\right),
\end{equation}
where $\kappa_\sL$ and $\epsilon_\sL$ are as in Table \ref{tab:LocalConditions}.
\end{proposition}

\begin{table}[ht]
\centering
\begin{tabular}{ccc}
\hline\vspace{-11pt}\\
$\sL$ & $\kappa_\sL$ & $\epsilon_\sL$ \\
\hline
\vspace{-7pt}\\
good & $\frac{q-1}{q}\frac{q^{10}}{q^{10}-1}$ & $q$\\
bad & $\frac{q^{9}-1}{q^{10}}\frac{q^{10}}{q^{10}-1}$ & $q$\\
multiplicative & $\frac{q-1}{q^2}\frac{q^{10}}{q^{10}-1}$ & $1$ \\
split multiplicative & $\frac{q-1}{2q^2}\frac{q^{10}}{q^{10}-1}$ & $1$ \\
nonsplit multiplicative & $\frac{q-1}{2q^2}\frac{q^{10}}{q^{10}-1}$  & $1$ \\
additive & $\frac{q^8-1}{q^{10}}\frac{q^{10}}{q^{10}-1}$  & $q$ \\
%
%$\textnormal{I}_m$ & $\frac{q^2-2q+1}{q^{m+2}}\frac{q^{10}}{q^{10}-1}$  & $q$ \\
%
%$\textnormal{II}$ & $\frac{q-1}{q^3}\frac{q^{10}}{q^{10}-1}$  & $1$ \\
%
%$\textnormal{III}$ & $\frac{q-1}{q^4}\frac{q^{10}}{q^{10}-1}$  & $1$ \\
%
%$\textnormal{IV}$ & $\frac{q-1}{q^5}\frac{q^{10}}{q^{10}-1}$  & $1$ \\
%
%$\textnormal{I}_0^\ast$ & $\frac{q-1}{q^6}\frac{q^{10}}{q^{10}-1}$  & $q$ \\
%
%$\textnormal{I}_m^\ast$ & $\frac{q^2-2q+1}{q^{m+7}}\frac{q^{10}}{q^{10}-1}$  & $q$ \\
%
%$\textnormal{II}^\ast$ & $\frac{q-1}{q^{10}}\frac{q^{10}}{q^{10}-1}$  & $1/q^2$ \\
%
%$\textnormal{III}^\ast$ & $\frac{q-1}{q^9}\frac{q^{10}}{q^{10}-1}$  & $1/q^3$ \\
%
%$\textnormal{IV}^\ast$ & $\frac{q-1}{q^8}\frac{q^{10}}{q^{10}-1}$  & $1/q$ \\
%
\hline
\end{tabular}
\caption{Local conditions}
\label{tab:LocalConditions}
\end{table}

For $b,c\in \F_q$, let $E_{b,c}$ denote the short Weierstrass model $y^2=x^3+bx+c$, with discriminant $\Delta(b,c)=-16(4b^3+27c^2)$ and trace of Frobenius $a_q(E_{b,c})$. Define the counting function
\begin{align}\label{eq:defH(a,q)}
H(a,q)\defeq\#\{(b,c)\in \F_q\times \F_q : \Delta(b,c)\neq 0,\ a_q(E_{b,c})=a\}.
\end{align}

\begin{proposition}[{\cite[Theorem 1.1.3]{Phi25}}]\label{prop:aqLocalCondition}
Let $K$ be a degree $d$ number field, and let $\p$ be a degree $n$ prime ideal of $\cO_K$ above a rational prime $p>3$. Set $q=N_{K/\Q}(\p)$. Let $a\in \Z$ be an integer satisfying $|a|\leq 2\sqrt{q}$. Let $\kappa(K)$ be the constant in Proposition \ref{prop:LocalConditions}. Then the number of elliptic curves over $K$ with naive height less than $B$, good reduction at $\p$, and which have trace of Frobenius $a$ at $\p$, is
\begin{equation}
\kappa(K) \frac{q^{8}}{q^{10}-1} H(a,q) B^{5/6}+O\left(\frac{H(a,q)}{q} B^{\frac{5}{6}-\frac{1}{3d}}\right).
\end{equation}
\end{proposition}

\begin{remark}\label{rem:finitely_many_local_conditions}
    Although Proposition \ref{prop:LocalConditions} and Proposition \ref{prop:aqLocalCondition} are stated for a local condition at a single place, they can be extended to local conditions at finitely many places by multiplying together the relevant densities, via \cite[Theorem 4.0.5]{Phi25}.
\end{remark}

We now collect some estimates on sums involving $H(a,q)$.

\begin{proposition}[{\cite[Proposition 6.2.2]{Phi25}}]\label{prop:ClassNumberSums}
Let $q$ be a power of a prime $p>3$. For any $\epsilon>0$ and $n\in \Z_{\geq0}$ we have the following estimates:
\begin{align}
\label{eq:H_sum_0} &\sum_{|a|\leq 2\sqrt{q}} H(a,q)= q^2-q,   \\
\label{eq:H_sum_2} &\sum_{|a|\leq 2\sqrt{q}} a^2 H(a,q)=q^3+O\left(q^{\frac{5}{2}+\epsilon}\right),\\
\label{eq:H_sum_odd} &\sum_{|a|\leq 2\sqrt{q}} a^{2n+1}H(a,q)=0, \text{ and }\\
\label{eq:H_sum_even} &\sum_{|a|\leq 2\sqrt{q}} a^{2n}H(a,q)\leq 2^{2n}q^{n+2}.
\end{align}
\end{proposition}

\begin{proof}
   The estimates \eqref{eq:H_sum_0} and \eqref{eq:H_sum_2} are proven in \cite[Proposition 6.2.2]{Phi25}.

   The estimate \eqref{eq:H_sum_odd} follows by pairing the terms $a$ and $-a$, and noting that $H(a,q)=H(-a,q)$.
   
    For the estimate \eqref{eq:H_sum_even}, noting that $a^{2n}\leq 2^{2n}q^n$ for $|a|\leq 2\sqrt{q}$, we have
    \begin{equation}
        \sum_{|a|\leq 2\sqrt{q}} a^{2n} H(a,q) \leq 2^{2n}q^n \sum_{|a|\leq 2\sqrt{q}} H(a,q).
        %\leq 2^{2n}q^{n+2}.
    \end{equation}
    Then the desired inequality follows from \eqref{eq:H_sum_0}.   
\end{proof}

\begin{remark}
    Although we will not need this, we note that the estimate \eqref{eq:H_sum_2} can be improved to the exact formula
    \begin{equation}
        \sum_{|a|\leq 2\sqrt{q}} a^2 H(a,q) = q^3 - q^2 - q + 1.
    \end{equation}
   This can be obtained by following the same argument as in \cite[Proposition 6.2.2]{Phi25} and using the additional fact that the space of weight $k$ cusp forms of level $1$ is zero dimensional for all $k<12$ (see, e.g., \cite[Chapter 3]{DS05}).
\end{remark}

\section{Frobenius trace formula for elliptic curves over number fields}\label{sec:Frobenius_trace}

In this section we use the results of Section \ref{sec:Counting_ECs} to estimate sums of the form
\begin{equation}\label{eq:frobenius_sum_intro}
    \sum_{E\in \cE_K(B)} \prod_{i=1}^t \widehat{a_E}(\p_{v_i}),
\end{equation}
which appear in the sums $S_1(j,\phi,B)$ defined in equation \eqref{eq:S_1_sum}. Our main estimate is Proposition \ref{prop:frobenius_trace}, which proves an analog of the Frobenius trace formula of Cho and Jeong \cite[Theorem 3.1]{CJ23}.

Let $t\in \Z_{>0}$ be a positive integer, and for each $i\in \{1,2,\dots,t\}$ let
\begin{itemize}
    \item $\p_i$ be a prime ideal of $K$ not lying above $2$ or $3$, with $\p_i\neq \p_{i'}$ whenever $i\neq i'$, and
    \item $s_i\in \Z_{>0}$.
\end{itemize}
Define $q_i\defeq N_{K/\Q}(\p_i)$ to be the norm of $\p_i$.

\begin{proposition}\label{prop:frobenius_trace}\hfill
\begin{enumerate}[label=\textnormal{(\roman*)}]
    \item If $s_i=2$ for all $i$, then
\begin{salign}\label{eq:frob_trace_i}
    &\sum_{E\in \cE_K(B)} \prod_{i=1}^t \widehat{a_E}(\p_i)^{s_i}\\
    &\quad = \left(\prod_{i=1}^t q_i\right) \#\cE_K(B)
    +O_\varepsilon\!\left(\left(\sum_{k=1}^t q_k^{1/2+\epsilon} \prod_{\substack{i=1 \\ i \neq k}}^t q_i\right) B^{5/6} + \left(\prod_{i=1}^t q_i^2\right) B^{5/6-1/3d} \right).
\end{salign}
    \item If $s_i$ is even for all $i$, then
\begin{salign}\label{eq:frob_trace_ii}
    \sum_{E\in \cE_K(B)} \prod_{i=1}^t \widehat{a_E}(\p_i)^{s_i}
    \ \ll\ \left(\prod_{i=1}^t q_i^{s_i/2}\right)B^{5/6} + \left(\prod_{i=1}^t q_i^{s_i/2+1}\right) B^{5/6-1/3d}.
\end{salign}
    \item If $s_i$ is odd for some $i$, then
\begin{salign}\label{eq:frob_trace_iii}
    \sum_{E\in \cE_K(B)} \prod_{i=1}^t \widehat{a_E}(\p_i)^{s_i}
    \ \ll\ \left(\prod_{i=1}^t q_i^{-1}\right) B^{5/6}
    + \left(\prod_{i=1}^t q_i^{s_i/2+1}\right) B^{5/6-1/3d}.
\end{salign}
\end{enumerate}
\end{proposition}

Before giving the proof, we record how the summand in \eqref{eq:frobenius_sum_intro} depends on the reduction type of $E$ at each $\p_i$. Recall from Section \ref{sec:L-functions} that
\begin{equation}\label{eq:a_hat_reduction_recall}
    \widehat{a_E}(\p_i)=
    \begin{cases}
        a_{\p_i}(E) & \text{ if $E$ has good reduction at $\p_i$,}\\
        b_{v_i} & \text{ if $E$ has bad reduction at $\p_i$.}
    \end{cases}
\end{equation}
%where $b_{v_i}\in\{+1,-1,0\}$ according as $E$ has split multiplicative, nonsplit multiplicative, or additive reduction at $\p_i$; see the definition of $b_v$ in Section \ref{sec:L-functions}. 
In particular, $|\widehat{a_E}(\p_i)|=1$ if $E$ has multiplicative reduction at $\p_i$, and $\widehat{a_E}(\p_i)=0$ if $E$ has additive reduction at $\p_i$.

For a subset $G\subseteq\{1,\dots,t\}$, write $M\defeq\{1,\dots,t\}\setminus G$, and let $\cE_K^{G,M}(B)$ denote the set of elliptic curves over $K$ of naive height at most $B$ that have good reduction at $\p_i$ for every $i\in G$ and multiplicative reduction at $\p_i$ for every $i\in M$. Since the reduction type of $E$ at each $\p_i$ is exactly one of good, multiplicative, or additive, and since the product in \eqref{eq:frobenius_sum_intro} vanishes as soon as $E$ has additive reduction at any $\p_i$, we may partition the sum \eqref{eq:frobenius_sum_intro} according to the set $G$ of indices at which $E$ has good reduction:
\begin{equation}\label{eq:reduction_partition}
    \sum_{E\in \cE_K(B)} \prod_{i=1}^t \widehat{a_E}(\p_i)^{s_i}
    = \sum_{G\subseteq\{1,\dots,t\}}\ \sum_{E\in \cE_K^{G,M}(B)} \prod_{i=1}^t \widehat{a_E}(\p_i)^{s_i}.
\end{equation}
The two extreme terms $G=\{1,\dots,t\}$ and $G=\emptyset$ correspond to good reduction at every $\p_i$ and to multiplicative reduction at every $\p_i$, respectively, while the other terms record the mixed cases in which $E$ has good reduction at some of the $\p_i$ and multiplicative reduction at the others.
The following lemma estimates the mixed contribution to \eqref{eq:reduction_partition}. 
%Part (a) is a general bound, used in Cases (i) and (ii) of Proposition \ref{prop:frobenius_trace}, in which the multiplicative places contribute a saving of $\prod_{i\in M}q_i^{-1}$ relative to the good-everywhere term. Part (b) records the stronger cancellation that occurs when some exponent $s_i$ is odd, and is used in Case (iii): whenever an odd exponent is present, every mixed pattern loses its main term altogether. Here the key arithmetic input is that split and nonsplit multiplicative reduction occur with equal leading density (see Table \ref{tab:LocalConditions}), so that the values $b_{v_i}=+1$ and $b_{v_i}=-1$ cancel when raised to an odd power and averaged over the family.

\begin{lemma}\label{lem:mixed_reduction}
    The following hold:
    \begin{enumerate}[label=\textnormal{(\alph*)}]
        \item In general,
        \begin{equation}\label{eq:mixed_reduction_bound}
            \left|\sum_{E\in \cE_K^{G,M}(B)} \prod_{i=1}^t \widehat{a_E}(\p_i)^{s_i}\right|
            \ \ll\
            \left(\prod_{i\in G} q_i^{s_i/2}\right)\left(\prod_{i\in M}\frac{1}{q_i}\right) B^{5/6}
            + \left(\prod_{i\in G} q_i^{s_i/2+1}\right) B^{5/6-1/3d}.
        \end{equation}
        \item If $s_{i_0}$ is odd for some $i_0\in\{1,\dots,t\}$, then 
        \begin{equation}\label{eq:mixed_reduction_bound_odd}
            \left|\sum_{E\in \cE_K^{G,M}(B)} \prod_{i=1}^t \widehat{a_E}(\p_i)^{s_i}\right|
            \ \ll\
            \left(\prod_{i\in G} q_i^{s_i/2+1}\right) B^{5/6-1/3d}.
        \end{equation}
    \end{enumerate}
\end{lemma}

\begin{proof}
    \underline{Part (a):} If $E\in \cE_K^{G,M}(B)$, then $E$ has multiplicative reduction at $\p_i$ for each $i\in M$, so that $\widehat{a_E}(\p_i)^{s_i}\in\{\pm 1\}$. Consequently
    \begin{equation}\label{eq:mixed_absorb_M}
        \left|\prod_{i=1}^t \widehat{a_E}(\p_i)^{s_i}\right|
        =\prod_{i\in G}\left|a_{\p_i}(E)\right|^{s_i}.
    \end{equation}
    Grouping elliptic curves in $\cE_K^{G,M}(B)$ according to their traces of Frobenius $a_{\p_i}(E)=a_i$ at the good places $i\in G$, we have
    \begin{equation}\label{eq:mixed_group_by_trace}
        \left|\sum_{E\in \cE_K^{G,M}(B)} \prod_{i=1}^t \widehat{a_E}(\p_i)^{s_i}\right|
        \leq
        \sum_{\substack{(a_i)_{i\in G}\\ |a_i|\leq 2\sqrt{q_i}}}
        \left(\prod_{i\in G}|a_i|^{s_i}\right)
        \#\!\left\{E\in \cE_K^{G,M}(B) : a_{\p_i}(E)=a_i\ \forall i\in G \right\}.
    \end{equation}
    The curves counted on the right of \eqref{eq:mixed_group_by_trace} are exactly those of height at most $B$ satisfying the local condition of good reduction with trace $a_i$ at $\p_i$ for $i\in G$ and multiplicative reduction at $\p_i$ for $i\in M$. Combining the prescribed--trace count of Proposition \ref{prop:aqLocalCondition} at the places indexed by $G$ with the multiplicative local condition of Proposition \ref{prop:LocalConditions} at the places indexed by $M$, and extending to finitely many places via Remark \ref{rem:finitely_many_local_conditions}, this count is
    \begin{equation}\label{eq:mixed_count_input}
        \kappa(K)
        \left(\prod_{i\in G}\frac{q_i^{8}}{q_i^{10}-1}H(a_i,q_i)\right)
        \left(\prod_{i\in M}\frac{q_i-1}{q_i^2}\frac{q_i^{10}}{q_i^{10}-1}\right)B^{5/6}
        +O\!\left(\left(\prod_{i\in G}\frac{H(a_i,q_i)}{q_i}\right) B^{5/6-1/3d}\right).
    \end{equation}
    %where we have used that the multiplicative condition contributes a factor $\epsilon_\sL=1$ to the error term (see Table \ref{tab:LocalConditions}).

    We first bound the main term of \eqref{eq:mixed_count_input}. By the geometric series, we have 
    \begin{equation}\label{eq:geometric_series}
        \frac{q_i^{10}}{q_i^{10}-1}=1+O(q_i^{-10}).
    \end{equation}
    For each $i\in M$ we have $\tfrac{q_i-1}{q_i^2}\ll q_i^{-1}$. Since each $s_i$ is a positive integer, using the bound \eqref{eq:H_sum_even} when $s_i$ is even, and using the trivial bound 
    \begin{equation}
        \sum_{|a_i|\leq 2\sqrt{q_i}}|a_i|^{s_i}H(a_i,q_i)\leq (2\sqrt{q_i})^{s_i}\sum_{|a_i|\leq 2\sqrt{q_i}}H(a_i,q_i)\ll q_i^{s_i/2+2}
    \end{equation}
    (via \eqref{eq:H_sum_0}) in general, for each $i\in G$ we obtain 
    \begin{equation}\label{eq:mixed_H_sum}
        \sum_{|a_i|\leq 2\sqrt{q_i}}|a_i|^{s_i}\frac{q_i^{8}}{q_i^{10}-1}H(a_i,q_i)\ \ll\ q_i^{s_i/2}.
    \end{equation}
    Combining \eqref{eq:mixed_group_by_trace}, the main term of \eqref{eq:mixed_count_input}, and \eqref{eq:mixed_H_sum}, the main-term contribution is
    \begin{equation}\label{eq:mixed_main}
        \ll \left(\prod_{i\in G} q_i^{s_i/2}\right)\left(\prod_{i\in M}\frac{1}{q_i}\right) B^{5/6}.
    \end{equation}
    For the error term of \eqref{eq:mixed_count_input}, we interchange the order of summation and product and apply equation \eqref{eq:H_sum_0} together with the bound
    \begin{equation}
        |a_i|^{s_i}\leq (2\sqrt{q_i})^{s_i}\ll q_i^{s_i/2},
    \end{equation} 
    giving
    \begin{equation}\label{eq:mixed_error}
        O\!\left(\sum_{\substack{(a_i)_{i\in G}\\ |a_i|\leq 2\sqrt{q_i}}}\left(\prod_{i\in G}|a_i|^{s_i}\frac{H(a_i,q_i)}{q_i}\right) B^{5/6-1/3d}\right)
        = O\!\left(\left(\prod_{i\in G} q_i^{s_i/2+1}\right) B^{5/6-1/3d}\right).
    \end{equation}
    The bound \eqref{eq:mixed_reduction_bound} then follows from \eqref{eq:mixed_main} and \eqref{eq:mixed_error}.

\noindent
\underline{Part (b):} Suppose $s_{i_0}$ is odd for some index $i_0$. We distinguish two cases according to whether $i_0$ lies in $G$ or in $M$.

    \emph{The case $i_0\in G$.} We return to the count \eqref{eq:mixed_count_input}, whose main term factors over $i\in G$. The factor indexed by $i_0$ is
    \begin{equation}\label{eq:mixed_odd_good_vanish}
        \sum_{|a_{i_0}|\leq 2\sqrt{q_{i_0}}} a_{i_0}^{s_{i_0}}\,\frac{q_{i_0}^{8}}{q_{i_0}^{10}-1}H(a_{i_0},q_{i_0})=0,
    \end{equation}
    by equation \eqref{eq:H_sum_odd}. Hence the entire main term of \eqref{eq:mixed_count_input} vanishes, and only the error term of \eqref{eq:mixed_count_input} survives. Estimating this error term exactly as in \eqref{eq:mixed_error} yields the bound \eqref{eq:mixed_reduction_bound_odd}.

    \emph{The case $i_0\in M$.} In this case we exploit the cancellation between split and nonsplit multiplicative reduction. Recall from the definition of $b_v$ in Section \ref{sec:L-functions} that 
    \begin{equation}
        b_{v_{i_0}} = 
        \begin{cases}
            1 & \text{ if $E$ has split multiplicative reduction at $\p_{i_0}$,}\\
            -1 & \text{ if $E$ has nonsplit multiplicative reduction at $\p_{i_0}$.}
        \end{cases}
    \end{equation}
    Since $s_{i_0}$ is odd, $b_{v_{i_0}}^{s_{i_0}}=b_{v_{i_0}}$. Splitting $\cE_K^{G,M}(B)$ according to whether the reduction at $\p_{i_0}$ is split multiplicative or nonsplit multiplicative, and writing $\cE_K^{G,M,\pm}(B)$ for the corresponding subsets, we have
    \begin{equation}\label{eq:mixed_odd_mult_split}
        \sum_{E\in \cE_K^{G,M}(B)} \prod_{i=1}^t \widehat{a_E}(\p_i)^{s_i}
        = \sum_{E\in \cE_K^{G,M,+}(B)} \prod_{\substack{i=1\\ i\neq i_0}}^t \widehat{a_E}(\p_i)^{s_i}
        - \sum_{E\in \cE_K^{G,M,-}(B)} \prod_{\substack{i=1\\ i\neq i_0}}^t \widehat{a_E}(\p_i)^{s_i}.
    \end{equation}
    By Table \ref{tab:LocalConditions}, split multiplicative and nonsplit multiplicative reduction at $\p_{i_0}$ occur with the same density. Consequently, grouping each of the two sums on the right of \eqref{eq:mixed_odd_mult_split} by the traces $a_{\p_i}(E)=a_i$ at the good places $i\in G$ and applying the combined local count of Proposition \ref{prop:aqLocalCondition}, Proposition \ref{prop:LocalConditions}, and Remark \ref{rem:finitely_many_local_conditions}, the two main terms are identical and cancel in the difference \eqref{eq:mixed_odd_mult_split}. What remains is bounded by the sum of the two error terms, which by Proposition \ref{prop:aqLocalCondition} is $O(B^{5/6-1/3d})$.  Estimating the remaining multiplicative places exactly as in \eqref{eq:mixed_error}, we find that  
    \begin{equation}
        \left|\sum_{E\in \cE_K^{G,M}(B)} \prod_{i=1}^t \widehat{a_E}(\p_i)^{s_i}\right|
        \ \ll\ \left(\prod_{i\in G} q_i^{s_i/2+1}\right) B^{5/6-1/3d}.
    \end{equation}
\end{proof}

\begin{proof}[Proof of Proposition \ref{prop:frobenius_trace}]
    Throughout, for $G\subseteq\{1,\dots,t\}$ we write $M=\{1,\dots,t\}\setminus G$, and we use the reduction partition \eqref{eq:reduction_partition}, which expresses the sum in question as a sum over $G$ of the contributions of $\cE_K^{G,M}(B)$. The two extreme terms $G=\{1,\dots,t\}$ (good reduction at every $\p_i$) and $G=\emptyset$ (multiplicative reduction at every $\p_i$) are treated directly below; all other cases are addressed using Lemma  \ref{lem:mixed_reduction}.

    \medskip
    \noindent
    \textbf{Case (i): ($s_i=2$ for all $i$).}\\
    \emph{Good reduction at every $\p_i$ (i.e., $G=\{1,\dots,t\}$).} If $E$ has good reduction at $\p_i$ then $\widehat{a_E}(\p_i)^2=a_{\p_i}(E)^2$. Grouping by traces of Frobenius and applying Proposition \ref{prop:aqLocalCondition} together with  Remark \ref{rem:finitely_many_local_conditions}, we have
    \begin{salign}\label{eq:case_i_good}
        &\sum_{\substack{a_1,\dots,a_t \\ |a_i|\leq 2\sqrt{q_i}}} \sum_{\substack{E\in\cE_K(B)\\ a_{\p_i}(E)=a_i}}  \prod_{i=1}^t \widehat{a_E}(\p_i)^{2}
        =\sum_{\substack{a_1,\dots,a_t \\ |a_i|\leq 2\sqrt{q_i}}}  \prod_{i=1}^t \sum_{\substack{E\in\cE_K(B)\\ a_{\p_i}(E)=a_i}}  a_{\p_i}(E)^{2} \\
        &\quad =\sum_{\substack{a_1,\dots,a_t \\ |a_i|\leq 2\sqrt{q_i}}}
        \left(\left(\prod_{i=1}^t a_i^2\,\kappa(K)\frac{H(a_i,q_i)}{q_i^2}\frac{q_i^{10}}{q_i^{10}-1}\right) B^{5/6}
        +O\!\left(\left(\prod_{i=1}^t \frac{H(a_i,q_i)}{q_i}\right) B^{5/6-1/3d}\right)\right)\\
        &\quad =\kappa(K) \left(\prod_{i=1}^t \frac{q_i^{8}}{q_i^{10}-1}
        \sum_{|a_i|\leq 2\sqrt{q_i}} a_i^2 H(a_i,q_i)\right) B^{5/6}
        +O\!\left(\left(\prod_{i=1}^t q_i^2\right) B^{5/6-1/3d}\right).
    \end{salign}
    By the estimate \eqref{eq:H_sum_2}, and by the geometric series identity \eqref{eq:geometric_series}, we have 
    \begin{equation}\label{eq:case_i_leading}
        \frac{q_i^{8}}{q_i^{10}-1}\sum_{|a_i|\leq 2\sqrt{q_i}} a_i^2 H(a_i,q_i)
        = q_i + O\!\left(q_i^{1/2+\epsilon}\right).
    \end{equation}
    Expanding the product of the factors \eqref{eq:case_i_leading} over $i$ and separating the leading term $\prod_i q_i$ from the remaining terms, the good--reduction contribution \eqref{eq:case_i_good} becomes
    \begin{equation}\label{eq:case_i_good_final}
        \kappa(K)\left(\prod_{i=1}^t q_i\right) B^{5/6}
        +O_\varepsilon\!\left(\left(\sum_{k=1}^t q_k^{1/2+\epsilon}\prod_{\substack{i=1\\ i\neq k}}^t q_i\right) B^{5/6}
        + \left(\prod_{i=1}^t q_i^2\right) B^{5/6-1/3d}\right).
    \end{equation}
    Since $\#\cE_K(B)=\kappa(K)B^{5/6}+O(B^{5/6-1/3d})$, the main term of \eqref{eq:case_i_good_final} equals $\left(\prod_{i=1}^t q_i\right)\#\cE_K(B)$ up to an error of size $O\!\left(\left(\prod_i q_i\right)B^{5/6-1/3d}\right)$, which is absorbed by the second error term of \eqref{eq:case_i_good_final}.

    \emph{Multiplicative reduction at every $\p_i$ (i.e., $G=\emptyset$).} If $E$ has multiplicative reduction at $\p_i$ then $\widehat{a_E}(\p_i)^2=b_{v_i}^2=1$. Hence, by the multiplicative local condition of Proposition \ref{prop:LocalConditions} and Remark \ref{rem:finitely_many_local_conditions},
    \begin{salign}\label{eq:case_i_mult}
        \left|\sum_{E\in \cE_K^{\emptyset,\{1,\dots,t\}}(B)} \prod_{i=1}^t \widehat{a_E}(\p_i)^{2}\right|
        &= \#\cE_K^{\emptyset,\{1,\dots,t\}}(B) \\
        &=\kappa(K) \left(\prod_{i=1}^t \frac{q_i-1}{q_i^2}\frac{q_i^{10}}{q_i^{10}-1}\right) B^{5/6}+O\!\left(B^{5/6-1/3d}\right) \\
        &=O\!\left(\left(\prod_{i=1}^t \frac{1}{q_i}\right) B^{5/6}\right).
    \end{salign}

    \emph{Mixed reduction.} For each subset $G\notin\{\emptyset, \{1,\dots,t\}\}$ of $\{1,\dots,t\}$, Lemma \ref{lem:mixed_reduction}(a), with $s_i=2$ for all $i$, gives
    \begin{equation}\label{eq:case_i_mixed}
        \left|\sum_{E\in \cE_K^{G,M}(B)} \prod_{i=1}^t \widehat{a_E}(\p_i)^{2}\right|
        \ \ll\ \left(\prod_{i\in G} q_i\right)\left(\prod_{i\in M}\frac{1}{q_i}\right) B^{5/6}
        + \left(\prod_{i\in G} q_i^{2}\right) B^{5/6-1/3d}.
    \end{equation}
    Since $M\neq\emptyset$, the main term on the right of inequality \eqref{eq:case_i_mixed} is at most $\left(\prod_{i\in G}q_i\right)q_{k}^{-1}\,B^{5/6}$ for any fixed $k\in M$, which is bounded by $\left(\sum_{k=1}^t q_k^{1/2+\epsilon}\prod_{i\neq k}q_i\right)B^{5/6}$; and the error term is bounded by $\left(\prod_{i=1}^t q_i^{2}\right)B^{5/6-1/3d}$. Thus each mixed term is absorbed into the two error terms of \eqref{eq:case_i_good_final}. 
    
    As there are at most $2^t=O_t(1)$ subsets $G$, summing the finitely many contributions \eqref{eq:case_i_good_final}, \eqref{eq:case_i_mult}, and \eqref{eq:case_i_mixed} over $G$ via the partition \eqref{eq:reduction_partition} yields
    \begin{equation}
        \sum_{E\in \cE_K(B)} \prod_{i=1}^t \widehat{a_E}(\p_i)^{2}
        = \left(\prod_{i=1}^t q_i\right)\#\cE_K(B)
        +O_\varepsilon\!\left(\left(\sum_{k=1}^t q_k^{1/2+\epsilon}\prod_{\substack{i=1\\ i\neq k}}^t q_i\right) B^{5/6}
        + \left(\prod_{i=1}^t q_i^{2}\right) B^{5/6-1/3d}\right).
    \end{equation}

    \medskip
    \noindent
    \textbf{Case (ii): ($s_i$ even for all $i$).}\\
    \emph{Good reduction at every $\p_i$ (i.e., $G=\{1,\dots,t\}$).} Grouping by traces of Frobenius and applying Proposition \ref{prop:aqLocalCondition} together with Remark \ref{rem:finitely_many_local_conditions}, we have
    \begin{salign}\label{eq:case_ii_good}
        \sum_{\substack{a_1,\dots, a_t \\ |a_i|\leq 2\sqrt{q_i}}} \sum_{\substack{E\in\cE_K(B)\\ a_{\p_i}(E)=a_i}}  \prod_{i=1}^t \widehat{a_E}(\p_i)^{s_i}
        &=
        \left(\prod_{i=1}^t
        \sum_{|a_i|\leq 2\sqrt{q_i}}a_i^{s_i}\,\kappa(K)\frac{H(a_i,q_i)}{q_i^2}\frac{q_i^{10}}{q_i^{10}-1}\right)
        B^{5/6}\\
        &\qquad +  O\!\left( \sum_{\substack{a_1,\dots,a_t \\ |a_i|\leq 2\sqrt{q_i}}}
       \left(\prod_{i=1}^t |a_i|^{s_i} \frac{H(a_i,q_i)}{q_i}\right)B^{5/6-1/3d}\right).
    \end{salign}
    Since $s_i$ is even for each $i$, the estimate \eqref{eq:H_sum_even}, together with equation \eqref{eq:geometric_series}, shows that the main term of \eqref{eq:case_ii_good} is
    \begin{equation}\label{eq:case_ii_good_main}
        O\!\left(\left(\prod_{i=1}^t q_i^{s_i/2}\right) B^{5/6}\right).
    \end{equation}
    For the error term of \eqref{eq:case_ii_good}, interchanging the order of the sum and the product and applying the bound \eqref{eq:H_sum_even} yields
    \begin{equation}\label{eq:case_ii_good_error}
        O\!\left(\left(\prod_{i=1}^t q_i^{s_i/2+1}\right) B^{5/6-1/3d}\right).
    \end{equation}

    \emph{Multiplicative reduction at every $\p_i$ (i.e., $G=\emptyset$).} As in Case (i), if $E$ has multiplicative reduction at every $\p_i$ then $|\widehat{a_E}(\p_i)^{s_i}|=1$, so by Proposition \ref{prop:LocalConditions} and Remark \ref{rem:finitely_many_local_conditions},
    \begin{equation}\label{eq:case_ii_mult}
        \left|\sum_{E\in \cE_K^{\emptyset,\{1,\dots,t\}}(B)} \prod_{i=1}^t \widehat{a_E}(\p_i)^{s_i}\right|
        = \#\cE_K^{\emptyset,\{1,\dots,t\}}(B)
        =O\!\left(\left(\prod_{i=1}^t \frac{1}{q_i}\right) B^{5/6}\right).
    \end{equation}

    \emph{Mixed reduction.} For each subset $G$ with $M\neq\emptyset$, Lemma \ref{lem:mixed_reduction}(a) gives
    \begin{equation}\label{eq:case_ii_mixed}
        \left|\sum_{E\in \cE_K^{G,M}(B)} \prod_{i=1}^t \widehat{a_E}(\p_i)^{s_i}\right|
        \ \ll\ \left(\prod_{i\in G} q_i^{s_i/2}\right)\left(\prod_{i\in M}\frac{1}{q_i}\right) B^{5/6}
        + \left(\prod_{i\in G} q_i^{s_i/2+1}\right) B^{5/6-1/3d}.
    \end{equation}
    Since $M\neq\emptyset$, the main term of inequality \eqref{eq:case_ii_mixed} is bounded by $\left(\prod_{i=1}^t q_i^{s_i/2}\right)B^{5/6}$, and its error term is bounded by $\left(\prod_{i=1}^t q_i^{s_i/2+1}\right)B^{5/6-1/3d}$. Hence each mixed term is dominated by the good--reduction bounds \eqref{eq:case_ii_good_main} and \eqref{eq:case_ii_good_error}. 
    
    Summing the $O_t(1)$ contributions over $G$ via the partition \eqref{eq:reduction_partition}, and noting that the multiplicative bound \eqref{eq:case_ii_mult} is likewise dominated by \eqref{eq:case_ii_good_main}, we obtain
    \begin{equation}
        \sum_{E\in \cE_K(B)} \prod_{i=1}^t \widehat{a_E}(\p_i)^{s_i}
        \ \ll\ \left(\prod_{i=1}^t q_i^{s_i/2}\right)B^{5/6} + \left(\prod_{i=1}^t q_i^{s_i/2+1}\right) B^{5/6-1/3d}.
    \end{equation}

    \medskip
    \noindent
    \textbf{Case (iii): ($s_i$ odd for some $i$).}
    Without loss of generality, suppose that $s_1$ is odd.

    \emph{Good reduction at every $\p_i$ (i.e., $G=\{1,\dots,t\}$).} Exactly as in the computation \eqref{eq:case_ii_good}, Proposition \ref{prop:aqLocalCondition} and Remark \ref{rem:finitely_many_local_conditions} give
    \begin{salign}\label{eq:case_iii_good}
        &\sum_{\substack{a_1,\dots, a_t \\ |a_i|\leq 2\sqrt{q_i}}} \sum_{\substack{E\in\cE_K(B)\\ a_{\p_i}(E)=a_i}}  \prod_{i=1}^t \widehat{a_E}(\p_i)^{s_i}\\
        &\quad =
        \left(\prod_{i=1}^t
        \sum_{|a_i|\leq 2\sqrt{q_i}}a_i^{s_i}\,\kappa(K)\frac{H(a_i,q_i)}{q_i^2}\frac{q_i^{10}}{q_i^{10}-1}\right)
        B^{5/6}
        +  O\!\left( \sum_{\substack{a_1,\dots,a_t \\ |a_i|\leq 2\sqrt{q_i}}}
       \left(\prod_{i=1}^t |a_i|^{s_i} \frac{H(a_i,q_i)}{q_i}\right)B^{5/6-1/3d}\right).
    \end{salign}
    Since $s_1$ is odd, estimate \eqref{eq:H_sum_odd} gives
    \begin{equation}\label{eq:case_iii_vanish}
        \sum_{|a_1|\leq 2\sqrt{q_1}}a_1^{s_1}H(a_1,q_1)\,\frac{q_1^8}{q_1^{10}-1}=0,
    \end{equation}
    so the $i=1$ factor of the main term of \eqref{eq:case_iii_good} vanishes, and hence the entire main term of \eqref{eq:case_iii_good} vanishes.

    For the error term of \eqref{eq:case_iii_good}, we bound $|a_i|^{s_i}\leq (2\sqrt{q_i})^{s_i}\ll q_i^{s_i/2}$ and interchange the order of the sum and the product:
    \begin{salign}\label{eq:case_iii_error}
        & O\!\left(\sum_{\substack{a_1,\dots,a_t \\ |a_i|\leq 2\sqrt{q_i}}}
        \left(\prod_{i=1}^t |a_i|^{s_i} \frac{H(a_i,q_i)}{q_i}\right)B^{5/6-1/3d}\right) \\
        &\qquad
        = O\!\left(\left(\prod_{i=1}^t q_i^{s_i/2-1}\sum_{|a_i|\leq 2\sqrt{q_i}} H(a_i,q_i)\right) B^{5/6-1/3d}\right).
    \end{salign}
    Applying the estimate \eqref{eq:H_sum_0}, the error term \eqref{eq:case_iii_error} is
    \begin{equation}\label{eq:case_iii_good_final}
        O\!\left(\left(\prod_{i=1}^t q_i^{s_i/2+1}\right) B^{5/6-1/3d}\right).
    \end{equation}

    \emph{Multiplicative reduction at every $\p_i$ (i.e., $G=\emptyset$).} If $E$ has multiplicative reduction at every $\p_i$ then $|\widehat{a_E}(\p_i)^{s_i}|=1$, so by Proposition \ref{prop:LocalConditions} and Remark \ref{rem:finitely_many_local_conditions},
    \begin{equation}\label{eq:case_iii_mult}
        \left|\sum_{E\in \cE_K^{\emptyset,\{1,\dots,t\}}(B)} \prod_{i=1}^t \widehat{a_E}(\p_i)^{s_i}\right|
        \leq \#\cE_K^{\emptyset,\{1,\dots,t\}}(B)
        \ll \left(\prod_{i=1}^t \frac{1}{q_i}\right) B^{5/6}+B^{5/6-1/3d}.
    \end{equation}

    \emph{Mixed reduction.} Since $s_1$ is odd, Lemma \ref{lem:mixed_reduction}(b) applies to every subset $G$ with $M\neq\emptyset$ and gives
    \begin{equation}\label{eq:case_iii_mixed}
        \left|\sum_{E\in \cE_K^{G,M}(B)} \prod_{i=1}^t \widehat{a_E}(\p_i)^{s_i}\right|
        \ \ll\ \left(\prod_{i\in G} q_i^{s_i/2+1}\right) B^{5/6-1/3d}
        \ \leq\ \left(\prod_{i=1}^t q_i^{s_i/2+1}\right) B^{5/6-1/3d}.
    \end{equation}
    In particular, each mixed contribution has no main term of order $B^{5/6}$ and is bounded by the second term of the asserted bound \eqref{eq:frob_trace_iii}.

    Summing the $O_t(1)$ contributions \eqref{eq:case_iii_good_final}, \eqref{eq:case_iii_mult}, and \eqref{eq:case_iii_mixed} over the subsets $G$ via the partition \eqref{eq:reduction_partition}, we obtain
    \begin{equation}
        \sum_{E\in \cE_K(B)} \prod_{i=1}^t \widehat{a_E}(\p_i)^{s_i}
        \ \ll\ \left(\prod_{i=1}^t q_i^{-1}\right) B^{5/6}
        + \left(\prod_{i=1}^t q_i^{s_i/2+1}\right) B^{5/6-1/3d}.
    \end{equation}
\end{proof}

\section{Bounding moments of analytic ranks of elliptic curves}\label{sec:moment_bound}

In this section we prove our bound for moments of analytic ranks of elliptic curves over number fields. To do this we prove an asymptotic for the moments of the $1$-level density of elliptic curves over number fields (Theorem \ref{thm:1-level_moment_bound}), which we do using results from Section \ref{sec:Frobenius_trace} to bound the sum $S_1(j,\phi,B)$ defined in equation \ref{eq:S_1_sum}.

\begin{lemma}\label{lem:S_1_j_odd}
    If $j$ is odd and the support of $\widehat{\phi}$ is contained in the interval $(-\nu, \nu)$, then
    \begin{equation}
        S_1(j, \phi, B) 
        \ll_\epsilon
        1 + B^{3j\nu/2-1/3d}\log(B)^j.
    \end{equation}
\end{lemma}

\begin{proof}
We may write $S_1(j,\phi,B)$ as
\begin{salign}
    S_1(j, \phi, B)
    &= \frac{1}{\#\cE_{K}(B)} \sum_{\substack{v_1,\dots,v_j\in \Val_0(K) \\ 2\nmid q_{v_i},\ 3\nmid q_{v_i}}} \prod_{i=1}^j \frac{\log(q_{v_i})}{q_{v_i}} \cdot \widehat{\phi}\left( \frac{\log(q_{v_i})}{\log(B)}\right)
\sum_{E\in\cE_{K}(B)}  \prod_{i=1}^j \widehat{a_E}(\p_{v_i})\\
     &= \frac{1}{\#\cE_{K}(B)} \sum_{s_1+\cdots+s_t = j}
     \sum_{\substack{v_1,\dots,v_t\in \Val_0(K) \\ v_i\neq v_{i'} \\ 2\nmid q_{v_i},\ 3\nmid q_{v_i}}} \prod_{i=1}^t \frac{\log(q_{v_i})^{s_i}}{q_{v_i}^{s_i}} \cdot \widehat{\phi}\left( \frac{\log(q_{v_i})}{\log(B)}\right)^{s_i}
\sum_{E\in\cE_{K}(B)}  \prod_{i=1}^t \widehat{a_E}(\p_{v_i})^{s_i}.
\end{salign}
Since $j$ is odd, at least one of the $s_i$ must be odd, and therefore, by Proposition \ref{prop:frobenius_trace}, we have that 
    \begin{salign}
       S_1(j,\phi, B) \ll_\epsilon \frac{1}{\#\cE_{K}(B)} \sum_{s_1+\cdots+s_t = j}\sum_{\substack{v_1,\dots,v_t\in \Val_0(K) \\ v_i\neq v_{i'} \\ 2\nmid q_{v_i},\ 3\nmid q_{v_i}}} \prod_{i=1}^t \frac{\log(q_{v_i})^{s_i}}{q_{v_i}^{s_i}} \cdot \widehat{\phi}\left( \frac{\log(q_{v_i})}{\log(B)}\right)^{s_i} \\
       \times \left( \left(\prod_{i=1}^t q_i^{-1}\right) B^{5/6}
    + \left(\prod_{i=1}^t q_i^{s_i/2+1}\right) B^{5/6-1/3d}\right).
    \end{salign}
    We see that the sum is dominated by the $s_1=s_2=\cdots=s_t=1$ term (i.e., when $t=j$), so that 
    \begin{salign}
        S_1(j,\phi,B) \ll_\epsilon \frac{1}{\#\cE_{K}(B)} \sum_{\substack{v_1,\dots,v_j\in \Val_0(K) \\ v_i\neq v_{i'} \\ 2\nmid q_{v_i},\ 3\nmid q_{v_i}}} \prod_{i=1}^j \frac{\log(q_{v_i})}{q_{v_i}} \cdot \widehat{\phi}\left( \frac{\log(q_{v_i})}{\log(B)}\right) \\
       \times \left( \left(\prod_{i=1}^j q_i^{-1}\right) B^{5/6}
    + \left(\prod_{i=1}^j q_i^{3/2}\right) B^{5/6-1/3d}\right).
    \end{salign}
  
    Since $\#\cE_K(B)\asymp B^{5/6}$, this quantity is
    \begin{salign}
        \ll_\epsilon \sum_{\substack{v_1,\dots,v_j\in \Val_0(K) \\ v_i\neq v_{i'} \\ 2\nmid q_{v_i},\ 3\nmid q_{v_i}}} \prod_{i=1}^j \frac{\log(q_{v_i})}{q_{v_i}} \cdot \widehat{\phi}\left( \frac{\log(q_{v_i})}{\log(B)}\right) \\
       \times \left( \left(\prod_{i=1}^j q_i^{-1}\right)
    + \left(\prod_{i=1}^j q_i^{3/2}\right) B^{-1/3d}\right).
    \end{salign}
      By the assumption that $\widehat{\phi}$ is supported on $(-\nu,\nu)$, this is
       \begin{salign}
        \ll_{\epsilon,\phi} \sum_{\substack{v_1,\dots,v_j\in \Val_0(K) \\ v_i\neq v_{i'} \\ 2\nmid q_{i},\ 3\nmid q_{i}\\ q_i\leq B^\nu}} \prod_{i=1}^j \frac{\log(q_{v_i})}{q_{v_i}} 
       \left( \left(\prod_{i=1}^j q_i^{-1}\right)
    + \left(\prod_{i=1}^j q_i^{3/2}\right) B^{-1/3d}\right).
    \end{salign}
    We have
    \begin{equation}
        \sum_{q\leq B^\nu} \frac{\log(q)}{q^2}=O(1)
    \end{equation}
    and by the prime ideal theorem, we have
    \begin{equation}
        \sum_{q\leq B^\nu} q^{1/2}\log(q)\sim B^{3\nu/2}\log(B).
    \end{equation}
    This leads to the bound
    \begin{salign}
        S_1(j,\phi,B) \ll_{\epsilon,\phi} 1 + B^{3j\nu/2-1/3d}\log(B)^j.
    \end{salign}
\end{proof}

\begin{lemma}\label{lem:S_1_j_even}
    If $j$ is even and the support $\widehat{\phi}$ is contained in the interval $(-\nu, \nu)$, then
    \begin{equation}
        S_1(j, \phi, B) 
        = \frac{j! \log(B)^{j}}{(j/2)! 2^{j}}  \left(\int_{\R} |u| \widehat{\phi}\left( u\right)^{2} du \right)^{j/2}
        + O_{\nu}\left(\log(B)^{j-2} + B^{3j\nu/2 - 1/3d}\log(B)^{j}\right).
    \end{equation}
\end{lemma}

\begin{proof}
    We may write $S_1(j,\phi,B)$ as
\begin{salign}
    & S_1(j, \phi, B)
    = \frac{1}{\#\cE_{K}(B)} \sum_{\substack{v_1,\dots,v_j\in \Val_0(K) \\ 2\nmid q_{v_i},\ 3\nmid q_{v_i}}} \prod_{i=1}^j \frac{\log(q_{v_i})}{q_{v_i}} \cdot \widehat{\phi}\left( \frac{\log(q_{v_i})}{\log(B)}\right)
\sum_{E\in\cE_{K}(B)}  \prod_{i=1}^j \widehat{a_E}(\p_{v_i})\\
     &\, = \frac{1}{\#\cE_{K}(B)} \sum_{s_1+\cdots+s_t = j} \frac{j!}{s_1! s_2!\cdots s_t!} \frac{1}{t!}
     \sum_{\substack{v_1,\dots,v_t\in \Val_0(K) \\ v_i\neq v_{i'} \\ 2\nmid q_{v_i},\ 3\nmid q_{v_i}}} \prod_{i=1}^t \frac{\log(q_{v_i})^{s_i}}{q_{v_i}^{s_i}} \cdot \widehat{\phi}\left( \frac{\log(q_{v_i})}{\log(B)}\right)^{s_i}
\sum_{E\in\cE_{K}(B)}  \prod_{i=1}^t \widehat{a_E}(\p_{v_i})^{s_i}.
\end{salign}
We decompose this sum into three pieces: 
\begin{enumerate}
    \item[(i)] $\Sigma_2$: $s_i=2$ for all $i$ (i.e., $s_i=\cdots=s_{j/2}=2$),
    \item[(ii)] $\Sigma_{even}$: $s_i$ even for all $i$, but $s_i\geq 4$ for some $i$, and
    \item [(iii)] $\Sigma_{odd}$: $s_i$ is odd for some $i$.
\end{enumerate}

\noindent
\textbf{Piece (i):} For this piece, applying Proposition \ref{prop:frobenius_trace}(i), we have 
\begin{salign}
    \Sigma_2 & = 
    \frac{1}{\#\cE_{K}(B)}\frac{j!}{(j/2)! 2^{j/2}} 
     \sum_{\substack{v_1,\dots,v_{j/2}\in \Val_0(K) \\ v_i\neq v_{i'} \\ 2\nmid q_{v_i},\ 3\nmid q_{v_i}}} \prod_{i=1}^{j/2} \frac{\log(q_{v_i})^{2}}{q_{v_i}^{2}} \cdot \widehat{\phi}\left( \frac{\log(q_{v_i})}{\log(B)}\right)^{2}
\sum_{E\in\cE_{K}(B)}  \prod_{i=1}^{j/2} \widehat{a_E}(\p_{v_i})^{2}\\
&= \frac{1}{\#\cE_{K}(B)} \frac{j!}{(j/2)! 2^{j/2}} 
     \sum_{\substack{v_1,\dots,v_{j/2}\in \Val_0(K) \\ v_i\neq v_{i'} \\ 2\nmid q_{v_i},\ 3\nmid q_{v_i}}} \prod_{i=1}^{j/2} \frac{\log(q_{v_i})^{2}}{q_{v_i}^{2}} \cdot \widehat{\phi}\left( \frac{\log(q_{v_i})}{\log(B)}\right)^{2}\\
     &\quad \times \left(\prod_{i=1}^{j/2} q_i\, \#\cE_K(B)
    +O_\varepsilon\left(\left(\sum_{k=1}^{j/2} q_k^{1/2+\epsilon} \prod_{\substack{i=1 \\ i \neq k}}^{j/2} q_i\right) B^{5/6} + \left(\prod_{i=1}^{j/2} q_i^2\right) B^{5/6-1/3d} \right)\right).
\end{salign}

Simplifying, and using that $\#\cE_K(B)\asymp B^{5/6}$ and $\widehat{\phi}$ has support contained in $(-\nu,\nu)$, we have
\begin{salign}\label{eq:messy_S_1_computation}
   \Sigma_2 
    & = \frac{j!}{(j/2)! 2^{j/2}}  \sum_{\substack{v_1,\dots,v_{j/2}\in \Val_0(K) \\ v_i\neq v_{i'} \\ 2\nmid q_{v_i},\ 3\nmid q_{v_i} \\ q_i\leq B^\nu}} \prod_{i=1}^{j/2} \frac{\log(q_{v_i})^{2}}{q_{v_i}} \cdot \widehat{\phi}\left( \frac{\log(q_{v_i})}{\log(B)}\right)^{2}\\
     &\quad + \frac{ j!}{(j/2)! 2^{j/2}}  \sum_{\substack{v_1,\dots,v_{j/2}\in \Val_0(K) \\ v_i\neq v_{i'} \\ q_i\leq B^\nu}} O_\varepsilon\left( \left(\sum_{k=1}^{j/2} \frac{\log(q_k)^2}{q_k^{3/2-\epsilon}} \prod_{\substack{i=1 \\ i \neq k}}^{j/2} \frac{\log(q_i)^2}{q_i}\right) + \prod_{i=1}^{j/2}\log(q_i)^2 B^{-1/3d} \right).
\end{salign}
Using Abel summation and the prime ideal theorem, we have
\begin{salign}
    \sum_{\substack{v_1,\dots,v_{j/2}\in \Val_0(K) \\ v_i\neq v_{i'} \\ 2\nmid q_{v_i},\ 3\nmid q_{v_i} \\ q_i\leq B^\nu}} \prod_{i=1}^{j/2} \frac{\log(q_{v_i})^{2}}{q_{v_i}} \cdot \widehat{\phi}\left( \frac{\log(q_{v_i})}{\log(B)}\right)^{2} 
    &= \prod_{i=1}^{j/2} \sum_{q_i\leq B^\nu} \frac{\log(q_{i})^{2}}{q_{i}} \cdot \widehat{\phi}\left( \frac{\log(q_{i})}{\log(B)}\right)^{2}\\
    &= \prod_{i=1}^{j/2} \int_1^\infty \frac{\log(t)}{t} \widehat{\phi}\left(\frac{\log(t)}{\log(B)}\right)^2 dt + O(1).
\end{salign}
Making the substitution $u=\log(t)/\log(B)$, we obtain
\begin{equation}
    \int_1^\infty \frac{\log(t)}{t} \widehat{\phi}\left(\frac{\log(t)}{\log(B)}\right)^2 dt 
    =  \log(B)^2 \int_0^\infty u \widehat{\phi}(u)^2 du 
    = \frac{\log(B)^2}{2} \int_\R |u| \widehat{\phi}(u)^2 du. 
\end{equation}
Substituting this into the main term of the asymptotic \eqref{eq:messy_S_1_computation}, we have
\begin{salign}\label{eq:mess_1}
    \frac{j!}{(j/2)!  2^{j/2}} \frac{\log(B)^{j}}{2^{j/2}}   \prod_{i=1}^{j/2} \int_{\R} |u| \widehat{\phi}\left( u\right)^{2} du
    = 
    \frac{j!\log(B)^j}{(j/2)!  2^{j}}  \left( \int_{\R} |u| \widehat{\phi}\left( u\right)^{2} du\right)^{j/2}.
\end{salign}
%\begin{salign}\label{eq:mess_1}
%    & \frac{\log(B)^{j}}{2^{j/2}}  \prod_{i=1}^{j/2} \int_{\R} |u| \widehat{\phi}\left( u\right)^{2} du \\
%     &\quad + \sum_{\substack{v_1,\dots,v_{j/2}\in \Val_0(K) \\ v_i\neq v_{i'} \\ q_i\leq B^\nu}} 
%     O_\varepsilon\left(  \left(\sum_{k=1}^{j/2} \frac{\log(q_k)^2}{q_k^{3/2-\epsilon}} \prod_{\substack{i=1 \\ i \neq k}}^{j/2} \frac{\log(q_i)^2}{q_i}\right) + \prod_{i=1}^{j/2}\log(q_i)^2 B^{-1/3d} \right).
%\end{salign}
We now address the error term. Arguments using the prime ideal theorem show that
\begin{equation}
    \sum_{q\leq B^\nu} \frac{\log(q)^2}{q} = O(\log(B^\nu)^2) \qquad  \text{ and }\qquad \sum_{q\leq B^\nu} \log(q)^2 = O( B^\nu \log(B^\nu)).
\end{equation}
For any $\epsilon<1/2$, we have
\begin{equation}
    \sum_{q\leq B^\nu} \frac{\log(q)^2}{q^{3/2-\epsilon}}= O(1).
\end{equation}
Therefore, the error term in \eqref{eq:messy_S_1_computation} is 
\begin{equation}
    O_{\epsilon,\nu}\left(\log(B)^{j-2} + B^{j\nu/2 - 1/3d}\log(B)^{j/2}\right).
\end{equation}

\noindent
\textbf{Piece (ii):} 
Using Proposition \ref{prop:frobenius_trace}(ii), we have
\begin{salign}
    \Sigma_{even} 
    &\ll_j \frac{1}{\#\cE_K(B)} 
    \sum_{\substack{s_1+\cdots+s_t=j\\ s_i \text{ all even}\\ \exists s_i\geq 4}}
    \sum_{\substack{v_1,\dots,v_{t}\in \Val_0(K) \\ v_i\neq v_{i'} \\ 2\nmid q_{v_i},\ 3\nmid q_{v_i}}} \prod_{i=1}^{t} \frac{\log(q_{v_i})^{s_i}}{q_{v_i}^{s_i}} \cdot \widehat{\phi}\left( \frac{\log(q_{v_i})}{\log(B)}\right)^{s_i}\\
    &\qquad \times \left(\left(\prod_{i=1}^t q_i^{s_i/2}\right)B^{5/6} + \left(\prod_{i=1}^t q_i^{s_i/2+1}\right) B^{5/6-1/3d}\right).
\end{salign}
Since $\#\cE_K(B)\asymp B^{5/6}$ and the support of $\widehat{\phi}$ is contained in $(-\nu,\nu)$, this simplifies to
\begin{equation}
    \Sigma_{even} \ll_j
     \sum_{\substack{s_1+\cdots+s_t=j\\ s_i \text{ all even}\\ \exists s_i\geq 4}} \left(\prod_{i=1}^t \sum_{q\leq B^\nu} \frac{\log(q)^{s_i}}{q^{s_i/2}}  + B^{-1/3d}\prod_{i=1}^t \sum_{q\leq B^\nu} \frac{\log(q)^{s_i}}{q^{s_i/2-1}}\right)
\end{equation}
By Abel summation and the prime ideal theorem, we have
\begin{equation}
    \sum_{q\leq B^\nu} \frac{\log(q)^{s_i}}{q^{s_i/2}} \ll 
    \begin{cases}
        \log(B)^2 & \text{ if } s_i=2,\\
        1 & \text{ if } s_i\geq 4,
    \end{cases}
\end{equation}
and
\begin{equation}
    \sum_{q\leq B^\nu} \frac{\log(q)^{s_i}}{q^{s_i/2-1}} \ll 
    \begin{cases}
        B^\nu\log(B) & \text{ if } s_i=2,\\
        \log(B)^5 & \text{ if } s_i=4,\\
        1 & \text{ if } s_i\geq 6.
    \end{cases}
\end{equation}
Since all $s_i$ are even and there exists some $i$ with $s_i\geq 4$ and $\sum s_i=j$, there can be at most $j/2-2$ indices with $s_i=2$. We thus have
\begin{salign}
    \Sigma_{even} 
    &\ll_j \log(B)^{2(j/2-2)} + B^{-1/3d}(B^\nu\log(B))^{j/2-2} \\
    &\ll_j \log(B)^{j-4} + B^{\nu(j/2-2)-1/3d}\log(B)^{j/2-2}.
\end{salign}

\noindent
\textbf{Piece (iii):} 
The same argument as in Proposition \ref{lem:S_1_j_odd} yields the bound
\begin{equation}
    \Sigma_{odd}\ll_{j,\nu} 1 + B^{3j\nu/2-1/3d}\log(B)^j.
\end{equation}
Combining the three pieces proves the lemma.
\end{proof}

We now consider the moments of the $1$-level density of elliptic curves. 

\begin{theorem}\label{thm:1-level_moment_bound}
    If the support of $\widehat{\phi}$ is contained in an interval $(-\nu,\nu)$ with $\nu\leq 2/9dm$, then
    \begin{equation}\label{eq:1-level_moment_bound}
        \frac{1}{\#\cE_{K}(B)} \sum_{E\in \cE_{K}(B)} D_1(E/K, \phi)^m 
       \leq  \sum_{k=0}^{\lfloor m/2\rfloor} \frac{m!}{(m-2k)!\, k!} \left(\widehat{\phi}(0) + \frac{\phi(0)}{2}\right)^{m-2k} \left(\int_\R |u| \widehat{\phi}(u)^2 du\right)^k + o(1). 
    \end{equation}    
\end{theorem}

\begin{proof}
    From equation \eqref{eq:1-level_moments_explicit_formula}, the left hand side of equation \eqref{eq:1-level_moment_bound} is bounded above by
    \begin{salign}\label{eq:1-level_moments}
        &  \sum_{j=0}^m \binom{m}{j} \left(\widehat{\phi}(0) + \frac{\phi(0)}{2}\right)^{m-j} \frac{(-2)^j}{\log(B)^j} S_1(j,\phi, B) \\
        &\ + O_m\left(\frac{\log\log(B)}{\log(B)} \sum_{j=0}^{m-1} \binom{m-1}{j} \left(\widehat{\phi}(0) + \frac{\phi(0)}{2}\right)^{m-1-j} \frac{(-2)^j}{\log(B)^j} S_1(j,\phi,B)\right).
    \end{salign}
    For $j$ odd, Lemma \ref{lem:S_1_j_odd}, together with the assumption $\nu\leq2/9dm$, implies that the terms of \eqref{eq:1-level_moments} with $j$ odd are $o(1)$. 

    For $j$ is even, let $k$ be such that $j=2k$. In this case, by Lemma \ref{lem:S_1_j_even}, together with the assumption $\nu\leq2/9dm$, we have
    \begin{equation}
        S_1(2k,\phi,B) 
        = \frac{(2k)!\log(B)^{2k}}{k! 2^{2k}} \left(\int_\R |u| \widehat{\phi}(u)^2 du\right)^k + O(\log(B)^{2k-2}).
    \end{equation}
    Substituting this into the expression \eqref{eq:1-level_moments}, we obtain
    \begin{equation}
        \sum_{k=0}^{\lfloor m/2\rfloor} \binom{m}{2k} \left(\widehat{\phi}(0) + \frac{\phi(0)}{2}\right)^{m-2k} \frac{(2k)!}{k!} \left(\int_\R |u| \widehat{\phi}(u)^2 du\right)^k+o(1).
    \end{equation}
    The theorem then follows from the identity
    \begin{equation}
        \binom{m}{2k}\frac{(2k)!}{k!} = \frac{m!}{(m-2k)!\, k!}.
    \end{equation}
\end{proof}

By the relation between the 1-level density and the analytic rank of elliptic curves (equation \eqref{eq:1-level_density}) Theorem \ref{thm:1-level_moment_bound} implies the following bound on moments of analytic ranks of elliptic curves:

\begin{corollary}\label{cor:moment_bound}
     If the support of $\widehat{\phi}$ is contained in an interval $(-\nu,\nu)$ with $\nu\leq 2/9dm$, then
    \begin{equation}
        \frac{1}{\#\cE_{K}(B)} \sum_{E\in \cE_{K}(B)} r_{an}(E)^m 
        \leq \sum_{k=0}^{\lfloor m/2\rfloor} \frac{m!}{(m-2k)!\, k!} \left(\frac{\widehat{\phi}(0)}{\phi(0)} + \frac{1}{2}\right)^{m-2k} \left(\frac{1}{\phi(0)^2}\int_\R |u| \widehat{\phi}(u)^2 du\right)^k + o(1). 
    \end{equation}    
\end{corollary}

To give a more explicit bound, we choose the test function
\begin{equation}\label{eq:phi}
\phi(y)=\begin{cases}
\left(\frac{\sin(\pi \nu y)}{2\pi y}\right)^2 & \text{ if } y\neq 0,\\
\nu^2/4 & \text{ if }y=0,
\end{cases}
\end{equation}
whose Fourier transform is
\begin{equation}
\widehat{\phi}(t)=\begin{cases}
\frac{1}{2}\left(\frac{\nu}{2}-\frac{|t|}{2}\right) &\text{ for } |t|\leq \nu\\
0 &\text{ for } |t|>\nu.
\end{cases}
\end{equation}
This test function satisfies the conditions of Proposition \ref{prop:modified-expicit-formula} (despite not being a Schwartz function). For this test function we have $\phi(0)=\nu^2/4$, $\widehat{\phi}(0)=\nu/4$, and
\begin{equation}
    \int_\R |u|\widehat{\phi}(u)^2 du = \frac{\phi(0)^2}{6}.
\end{equation}
Plugging these values into Corollary \ref{cor:moment_bound} gives Theorem \ref{thm:moment_bound}.

%TODO: remark possible improvements in $K=\Q$ case using ... But that this would still be difficult to carry out for general number fields. 

\section{Bounding the proportion of elliptic curves with large rank}

In this section we  obtain a bound on the proportion of elliptic curves with analytic rank larger than some bound.

\begin{proof}[Proof of Theorem \ref{thm:probability_large_rank}.]
By Proposition \ref{prop:1-level_density_expression} and the definition of the $1$-level density \eqref{eq:1-level_density}, we have
\begin{equation}
    r_{an}(E)\leq \frac{\widehat{\phi}(0)}{\phi(0)} + \frac{1}{2} - \frac{2}{\phi(0)\log(B)} U_1(\phi,E,B) + O\left(\frac{\log\log(B)}{\log(B)}\right).
\end{equation}
Choosing the test function $\phi$ defined in \eqref{eq:phi}, our inequality becomes
\begin{equation}
    r_{an}(E)\leq \frac{1}{\nu} + \frac{1}{2} - \frac{8}{\nu^2\log(B)} U_1(\phi,E,B) + O\left(\frac{\log\log(B)}{\log(B)}\right).
\end{equation}
Now let $\beta\in \R_{>0}$ and assume $r_{an}(E)\geq \beta$. Then, for sufficiently large $B$, 
\begin{equation}
    \frac{\nu^2}{4}\left(\beta-\frac{1}{\nu} - \frac{1}{2}\right) \leq - \frac{2}{\log(B)} U_1(\phi,E,B).
\end{equation}
We now rewrite this expression in a way that we can apply Lemma \ref{lem:S_1_j_even}. 
Letting $m\in \Z_{>0}$, we have
\begin{equation}
    \#\{E\in \cE_K(B) : r_{an}(E)\geq \beta\} \left(\frac{\nu^2(2\beta- 2/\nu - 1)}{8}\right)^{2m} 
    \leq \sum_{E\in \cE_K(B)}\left(- \frac{2}{\log(B)} U_1(\phi,E,B)\right)^{2m}.
\end{equation}
Rearranging terms and recalling the definition of $S_1(j,\phi,B)$, for sufficiently large $B$ (depending on $m$), we have
\begin{equation}
    \overline{\bP}_K(r_{an}(E)\geq \beta) \leq \left(\frac{8}{\nu^2(2\beta- 2/\nu - 1)}\right)^{2m} \left(\frac{2}{\log(B)}\right)^{2m} S_1(2m,\phi,B).
\end{equation}
Taking $\nu=2/9dm$ and applying Lemma \ref{lem:S_1_j_even}, this becomes
\begin{salign}
     \overline{\bP}_K(r_{an}(E)\geq \beta) 
     &\leq \left(\frac{8}{\nu^2(2\beta- 2/\nu - 1)}\right)^{2m} \frac{(2m)!}{m!} \left(\frac{\phi(0)^2}{6}\right)^{m}\\
     & = \frac{(2m)!}{m!}\left(\frac{2}{3 (2\beta - 2/\nu -1)^2}\right)^{m}\\
     & = \frac{(2m)!}{m!}\left(\frac{2}{3(2\beta - 9dm -1)^2}\right)^{m}.
\end{salign}
We now determine an asymptotic for this expression. 
Using Stirling's approximation, we have
\begin{equation}
	\frac{(2m)!}{m!} \ll \left(\frac{4m}{e}\right)^m,
\end{equation}
and therefore
\begin{equation}
\frac{(2m)!}{m!}\left(\frac{2}{3(2\beta - 9dm -1)^2}\right)^{m}\ll \left(\frac{8m}{3e(2\beta-9dm-1)^2}\right)^m.
\end{equation}
Choose 
\begin{equation}
    m=\left\lfloor\frac{2\beta}{9d} - \frac{\beta}{\log(\beta)}\right\rfloor. 
\end{equation}
Then, for $\beta$ sufficiently large (so that $m\ge 1$), we have 
\begin{equation}
	\frac{2\beta}{9d} - \frac{\beta}{\log(\beta)}-1 < m \leq \frac{2\beta}{9d}.
\end{equation}
Applying these bounds for $m$ we have
\begin{equation}
	\frac{8m}{3e(2\beta-9dm-1)^2}
	\leq \frac{16\beta/9d}{3e\left(\frac{9d\beta}{\log(\beta)}-1\right)^2}.
\end{equation}
For sufficiently large $\beta$ this is bounded above by 
\begin{equation}
	\frac{16\beta/9d}{3e\left(\frac{8d\beta}{\log(\beta)}\right)^2}= \frac{\log(\beta)^2}{108ed^3\beta}.
\end{equation}
From this and our lower bound for $m$, we have
\begin{equation}
	\overline{\bP}_K(r_{an}(E)\geq \beta) 
	\ll \left(\frac{\log(\beta)^2}{108ed^3\beta}\right)^{m} 
	\ll \left(\frac{\log(\beta)^2}{108ed^3\beta}\right)^{\frac{2\beta}{9d}-\frac{\beta}{\log(\beta)}-1}. 
\end{equation}
Further simplifying this bound, we have
\begin{salign}	\left(\frac{\log(\beta)^2}{108ed^3\beta}\right)^{\frac{2\beta}{9d}-\frac{\beta}{\log(\beta)}-1}
	& \ll \beta^{-\frac{2\beta}{9d}+\frac{\beta}{\log(\beta)}+1} \left(\frac{\log(\beta)^2}{108ed^3}\right)^{\frac{2\beta}{9d}}\\
	& \ll \beta^{-\frac{2\beta}{9d}+o(\beta)} \beta^{\frac{4\beta\log\log(\beta)}{9d\log(\beta)}}
	\beta^{\frac{2\beta\log(108ed^3)}{9d\log(\beta)}}\\
	&\ll_d \beta^{-\frac{2\beta}{9d}+o(\beta)}.
\end{salign}
Therefore
\begin{equation}
	\overline{\bP}_K(r_{an}(E) \geq \beta) \ll_d \beta^{-\frac{2\beta}{9d}+o(\beta)}.
\end{equation} 
\end{proof}

%\section{Final discussion}

%TODO: discuss limitations of methods used. possiblefuture improvements / challenges to alternative methods.

\bibliographystyle{alpha}
\bibliography{bibfile}

\end{document}